\documentclass[12pt,a4]{article}
\usepackage[colorlinks=true, pdfstartview=FitV, linkcolor=blue, citecolor=magenta,
urlcolor=blue]{hyperref}
\usepackage{color}
\usepackage[applemac]{inputenc}  
\usepackage{amssymb,amsmath,amsthm,amsfonts,amscd,amstext,stmaryrd,mathabx}
\usepackage[mathscr]{eucal}
\usepackage{array}
\usepackage{graphicx}
\usepackage{xypic}
\newtheorem{lem}{Lemma}[section]
\newtheorem{prop}[lem]{Proposition}
\newtheorem{thm}[lem]{Theorem}
\newtheorem{cor}[lem]{Corollary}
\newtheorem{df}{Definition}[section]
\newtheorem{exmpl}{Example}[section]
\newtheorem{rem}[lem]{Remark}

\newcommand{\Aut}{\mathrm{Aut}}
\def\aut{\operatorname {Aut}} 
\def\C{{\mathbb C}}  
\def\deg  {{\rm deg}}
\def\dim{\operatorname{dim}}  
\def\E{{\mathbb E}}    
\def\Hom{\operatorname{Hom}} 
\def\id{\operatorname{id}}  
\def\ker{\operatorname{ker}}  
\def\mod{\operatorname{mod}} 
\def\spec{\operatorname{Spec}}
\def\specf{\operatorname{Specf}}
\def\supp{\operatorname{supp}}
\def\OO{\cal O}  
\def\pr{\operatorname{pr}} 
\def\Q{{\mathbb Q}}   
\def\Z{{\mathbb Z}}    
\def\N{{\mathbb N}}    
\title{\bf Almost Commutative Probability Theory}
\author{Roland Friedrich and John McKay}
\begin{document}
\maketitle
\begin{abstract}
We solve two longstanding major problems in Free Probability. This is achieved by generalising the theory to one with values in arbitrary commutative algebras. We prove the existence of the multi-variable $S$-transform, and show that it is naturally realised as a faithful linear representation. Further, we prove that in dimension one, the analog of the classical relation between addition and multiplication of independent random variables holds for free random variables, if the co-domain is an algebra over the rationals. In this case the multiplicative problem can be reduced to the additive one, which is not true in dimensions greater than one.
Finally, we classify the groups which arise as joint distributions of $n$-tuples of non-commutative random variables, endowed with the free convolution product, which is the binary operation that encodes the multiplication of free $n$-tuples.   
\end{abstract}
\section{Introduction}
Free probability, which was originated by Voiculescu~\cite{VDN}, is a non-commutative probability theory with the central notion of freeness, instead of the well-known concept of independence from classical probability. 

In order to get a better understanding, let us illustrate both notions, which are algebraic. Roughly speaking, a probability space can be thought of as given by two, not-necessarily commutative, algebras $\mathcal{A},\mathcal{B}$, over some unital ring $k$, and a $k$-linear functional $\phi$, which is unital, i.e. $\phi:\mathcal{A}\rightarrow \mathcal{B}$ such that $\phi(1_{\mathcal{A}})=1_{\mathcal{B}}$. The elements of $\mathcal{A}$ are called $\mathcal{B}$-valued random variables. 

However, $\phi$ will in general not respect the given multiplicative structure, but for certain pairs of elements, it will do so. In classical probability, independence characterises those  elements whose product factorises with respect to the linear functional. 

On the other hand,  the property of freeness is related to the kernel of the linear functional $\phi$, and it requires certain products of elements of the kernel to remain therein.

It is customary to associate to a random variable $x\in\mathcal{A}$ an infinite, $\mathcal{B}$-valued vector, called its $\bullet$-transform, up to specification of $\bullet$, such as e.g. its distribution which would be given by its moment series $\mathcal{M}(x)\in\mathcal{B}^{\N}$.

Voiculescu solved in the 1980s the one-dimensional addition and multiplication problem for free random variables. He introduced the binary operations, $\boxplus_V$ and $\boxtimes_V$, on the space of distributions, and showed that in each case the resulting structure is that of an infinite-dimensional complex Lie group. 

He overcame the implicit and unfamiliar nature of the definitions of his operations, hardly accessible to direct computations, when he found two group isomorphisms onto the set of one-variable formal power series with the usual operations, namely addition and multiplication. These isomorphisms are called the $\mathcal{R}_V$-transform and the $\mathcal{S}_V$-transform, respectively, and are shown in Fig.~\ref{figure_moments}. 
\begin{figure}
\begin{center}
\[
\begin{xy}
  \xymatrix{
  (\mathcal{A},+,\cdot)\ar[d]_{\mathcal{M}:~\text{moments}}^{\text{additive}}\ar[rr]^{\mathcal{M}:~\text{moments}}_{\text{multiplicative}}&&{(\C^{\N},\boxtimes_V)_{\text{implicit}}}\ar[d]^{\mathcal{S}_V}\\
{\left(\C^{\N},\boxplus_V\right)_{\text{implicit}}}\ar[r]^{\mathcal{R}_V}&\left(\C^{\N},+\right)_{\text{explicit $\&$ familiar}}\ar@{-->}[r]^{\exists?}& \left(\C^{\N},\cdot\right)_{\text{explicit $\&$ familiar}}\ar@{-->}[l]
     }
\end{xy}
\]
\caption{The one-variable addition and multiplication problem for pairs of free non-commutative random variables, in the moment picture, as solved by Voiculescu~\cite{VDN}.}
\label{figure_moments}
\end{center}
\end{figure}

Since then, motivated by the classical case, the relation, if any, between the one-dimensional additive and multiplicative convolution, remained open, cf.~\cite{S97,NS}. 

Speicher~\cite{S97}, made the far-reaching discovery that an alternative, purely combinatorial description of free probability can be given, with the central ingredients being non-crossing partitions and free cumulants.

In the multi-variable case, the solutions to the addition and multiplication problem were found by Speicher and Nica~\cite{S97,N,NS}, based on the use of free cumulants. Further, a transformation, which involves the free cumulants was introduced, and which assigns a power series in $n$ non-commuting variables with complex coefficients to an $n$-tuple of non-commutative random variables. This map is called the $\mathcal{R}$-transform.

As with the one-dimensional case, they proved that the sets arising as $\mathcal{R}$-transforms, can be endowed with two binary operations which encode the addition and multiplication of free $n$-tuples of non-commutative random variables, such that they become in general non-abelian groups. However, contrary to the original situation for moments, they provided explicit and computable formulae for both operations, and also gave explicit transformation rules between moments and free cumulants, see~Figure~\ref{figure_cumulant}.

In the additive case the group structure is simply the addition of the $\mathcal{R}$-transforms, and in the multiplicative case, it is given by the so-called boxed convolution of the $\mathcal{R}$-transforms, which involves summation over non-crossing partitions. Even though this combinatorial operation is completely explicit and concrete, the desire to find an isomorphism onto a well-known group with an explicit multiplication rule, remained. This elusive isomorphism is commonly considered as the generalisation of the $S$-transform. 

Therefore, the problem of finding the higher-dimensional $S$-transform can be stated as the problem of understanding the group structure arising from the boxed convolution. 

Mastanak and Nica~\cite{MN1,MN2} approached this problem recently, by using Hopf algebraic methods in order to study the boxed convolution. In the one-dimensional case they establish, amongst other things, the connection with the Hopf algebra of symmetric functions, but the higher-dimensional case still remains unclear.

The starting point for our investigations is the realisation of a surprising connection between Witt vectors and Free Probability~\cite{FMcK1,FMcK2}, which suggests a more general, categorical approach.  
\begin{figure}
\begin{center}
\[
\begin{xy}
  \xymatrix{
  (\mathcal{A},+,\cdot)\ar[d]_{\mathcal{R}:~\text{free cumulants}}^{\text{additive}}\ar[rrr]^{\mathcal{R}:~\text{free cumulants}}_{\text{multiplicative}}&&&{(\C^{\N},\boxtimes)_{\text{explicit, but not familiar}}}\ar@{-->}[d]^{\exists\mathcal{S}?}\\
{\left(\C^{\N},+\right)_{\text{explicit $\&$ familiar}}}&&&\left(\C^{\N},\text{$\exists$ explicit $\&$ familiar operation?}\right)
     }
\end{xy}
\]
\caption{The multi-variable addition and multiplication problem for $n$-tuples of free random variables, in the free cumulant picture, as solved by Speicher and Nica~\cite{S97,N,NS}.}
\label{figure_cumulant}
\end{center}
\end{figure}
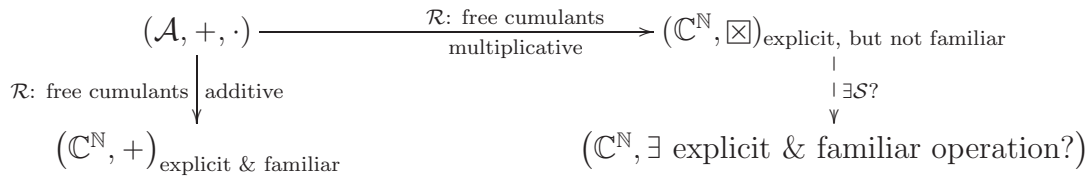
Our main tool are affine group schemes, which we apply to the generalisation of the usual  theory. 

Instead of the purely complex valued setting, we take as domain arbitrary algebras over a commutative unital ring $k$, and as co-domain, arbitrary commutative algebras over $k$. 
In the first place, this is possible, because the groups arising from the combinatorial description are algebraic, i.e. defined by polynomials with integer coefficients. 

Generalisations of the complex-valued situation have already been earlier developed~\cite{VDN,S98}, namely as the theory of ``free products with amalgamation" (FPwA). In that setting, one considers arbitrary, not-necessarily commutative, algebras, e.g. operators, as target. This ``doubly" non-commutative theory is however less understood. 

If we specialise FPwA to algebras over commutative algebras, then it reduces, as it is defined, to maps into the ground ring, which corresponds to a special case in our framework. In particular, in the more general setting one has to take care of some fine points, as e.g. the original notion of freeness has to be supplemented.   

We give answers to all the questions mentioned above, so that the question marks in Figures~\ref{figure_moments} and~\ref{figure_cumulant}, can be removed. 

A word concerning the style of this article is in place; it is written fairly self-contained, with examples and explicit proofs, up to the point were the connection with the relevant theorems in the literature is made.  

Also, for the moment we leave out the generalisation of the analytic aspects, which can be naturally incorporated into this framework, and further the discussion of what is called ``free probability of type $B$", which simply corresponds to certain rings with nilpotent elements.

Finally, we have to explain the choice of the title. ``Almost commutative" is used in different algebraic contexts, but it always expresses the fact that on a ``scale of non-commutativity" one might consider it as the ``first" step away from the abelian case. This will indeed become justified, after we have characterised the groups.

\section{Preliminaries and examples}
\label{PreEx}
In the following we shall assume $k$ to be at least a commutative ring with unit, unless further specified. Let $\mathbf{Set}$ denote the category of sets, $\mathbf{Grp}$ the category of groups, $\mathbf{cRing}$ the category of commutative unital rings, $\mathbf{Alg}_k$ the category of unital $k$-algebras and $\mathbf{cAlg}_k$ the category of commutative $k$-algebras. For $R\in\mathbf{cRing}$, we denote by $R^{\times}$ the multiplicative {group of units} (invertible elements) of $R$.

Consider for an integer $s\geq1$, the alphabet $[s]:=\{1,\dots,s\}$. A word $w$ is  a finite sequence (string) $(i_1\dots i_n)$ of elements $i_j\in[s]$, called letters. The set of all finite words over the alphabet $[s]$, including the empty word $\emptyset$, is denote by $[s]^*$. We define $[s]^*_+$ as $[s]^*\setminus\emptyset$, and for $n\in\N$,  let $[s]^*_n:=\{w\in[s]^*~|~|w|\leq n\}$ and $([s]^*_+)_n:=\{w\in[s]_+^*~|~1\leq|w|\leq n\}$, respectively.  

The set $[s]^*$ is a {\em monoid} with the multiplication $\circ$ given by concatenation of words, i.e. $( i_1\dots i_n)\circ(j_1\dots j_m):=(i_1\dots i_n j_1\dots j_m)$, and 
unit $1$, corresponding to the empty word $\emptyset$. For $n\in\N^{\times}$, the length $|w|$ of a word $w=(i_1\dots i_n)$ is $n$ and otherwise $|\emptyset|=0$.
The set $[s]^*$ is countable and graded by word length.

Let us first discuss the {\bf non-commutative power series} with commutative coefficients.

For $s\geq2$ and $R\in\mathbf{cRing}$, we denote by $R\langle\langle z_1,\dots, z_s\rangle\rangle$ the set of all formal power series in $s$ non-commuting indeterminates 
$\{z_1,\dots, z_s\}$ with coefficients in $R$, and by $R\langle z_1,\dots, z_s\rangle$ the respective non-commutative polynomials.

To every word $w=(i_1\dots i_n)\in\{1,\dots s\}^n$, $n\in\N^{\times}$, 
corresponds a {\bf monomial} 
$$
z_w:=z_{i_1}\cdots z_{i_n}\qquad\text{with $i_j\in\{1,\dots, s\}$ for $j=1,\dots n$},
$$
and the empty word is associated with the unit, i.e. $z_{\emptyset}:=1_R$. 

The {\bf degree}  of a monomial $z_w$ is 
$
\deg(z_w):=|w|.
$
A generic power series $f=f(z_1,\dots, z_s)$, with coefficients in $R$, can be written as
 \begin{equation}
 f(z_1,\dots z_s)=\sum_{w\in[s]^*} \alpha_w z_w=\alpha_0+\sum_{n=1}^{\infty}\sum_{i_1,\dots, i_n=1}^{\infty}\alpha_{i_1,\dots,i_n} z_{i_1}\cdots z_{i_n}~,
 \end{equation}
where $a_w$ or $a_{(i_1\dots i_n)}$ denotes the coefficient of $z_w$. Alternatively, we shall use $f_w$ to denote the coefficient corresponding to $z_w$.

For $s\geq2$, the set $R\langle\langle z_1,\dots, z_s\rangle\rangle$ has a natural structure of a non-commutative $R$-algebra.  The $R$-module part is defined by component-wise addition and scalar multiplication, and the product of two formal series $f\cdot g$ is given by 
\begin{equation}
\label{noncom_prod}
(f\cdot g)_w:=\sum_{\substack{w_1,w_2\in[s]^*\\w_1\circ w_2=w}}f_{w_1}\cdot_R g _{w_2}.
\end{equation}

For a single letter alphabet, we obtain the set of one-variable commutative formal power series $R[[z]]$ and polynomials $R[z]$, respectively.

We denote by $R_+\langle\langle z_1,\dots, z_s\rangle\rangle$ the set of {\bf power series without constant term}, which is a free $R$-module with point-wise addition and scalar multiplication.

Let us note that two other equivalent descriptions are as infinite strings or trees.

\subsection{Group schemes}
\label{Sec:GSch}
Let us recall basic facts from the theory of algebraic and formal groups. References are~\cite{D72,Froh,H1,Hum,Mil}.
 
Let $\mathcal{C}$ be a category. Every $A\in\operatorname{obj}(\mathcal{C})$ defines a functor $h^A:\mathcal{C}\rightarrow\mathbf{Set}$ by $h^A:=\Hom_{\mathcal{C}}(A,-)$. A functor $F:\mathcal{C}\rightarrow\mathbf{Set}$ is called {\bf representable}, if $F$ is {\bf naturally isomorphic} to some $h^A$, i.e. there exists a natural transformation $\tau:h^A\rightarrow F$. By the {\bf Yoneda Lemma}, this is equivalent to specifying an element $a\in F(A)$ and then the pair $(A,a)$ is said {\bf to represent $F$}. 

The general situation is now specialised to the category $\mathbf{cAlg}_k$.
\begin{df}
A functor $X:\mathbf{cAlg}_k\rightarrow\mathbf{Set}$, is called a {\bf $k$-functor}. For $A,R\in\mathbf{cAlg}_k$, the functor $h^A:=\Hom_{\mathbf{cAlg}_k}(A,-)$, is denoted as $\spec_k(A)$, and the set $\spec_k(A)(R)$, is called the {\bf spectrum of $R$}. 
\end{df}
\begin{df}
An {\bf affine $k$-scheme} is a $k$-functor which is representable by some $\spec_k(A)$, i.e. naturally isomorphic. It is {\bf algebraic} if the $k$-algebra $A$ is {\em finitely} generated, and {\bf connected} if $A$ is a local ring. 
\end{df}

\begin{df}
A $k$-functor $G:\mathbf{cAlg}_k\rightarrow\mathbf{Grp}$, is called a {\bf $k$-group functor}, and $G$ is an {\bf affine group scheme}, if it is affine as a $k$-functor.
\end{df}

The elements in $\spec_k(A)(R)$ are $k$-algebra morphisms, i.e. for $\lambda\in k$ and $x,y\in A$, we have
$$
\varphi:A\rightarrow R,\qquad \varphi(\lambda xy)=\lambda\varphi(x)\varphi(y)\quad\text{and}\quad \varphi(1_A)=1_R.
$$
If $A$ is a Hopf algebra, an associative product can be defined on $\spec_k(A)(R)$, called the {\bf convolution product}. For $\varphi_1,\varphi_2\in\spec_k(A)(R)$, define
$$
\varphi_1\ast\varphi_2:=\mu\circ(\varphi_1\otimes\varphi_2)\circ\Delta,
$$
as given by the sequence of maps
\[
\begin{CD}
A@>\Delta>>A\otimes_k A@>\varphi_1\otimes\varphi_2>>R\otimes_k R@>\mu>>R.
\end{CD}
\]
The unit element for the convolution product is $\eta\circ\varepsilon$, i.e. the composition of the co-unit with the unit, as shown below
\[
\begin{CD}
A@>\varepsilon>>k@ >\eta>>A.
\end{CD}
\]
More generally the following duality holds:
\begin{prop}[cf. e.g. \cite{D72,Mil,W}]
Hopf algebras over $k$ and affine $k$-group schemes are in bijective correspondence.
\end{prop}
\begin{exmpl}
\end{exmpl}
\begin{enumerate}
\item The {\bf additive group} $\mathbb{G}_a$ is the functor $R\rightarrow R$, (additive group) $R\in\mathbf{cAlg}_k$, represented by the Hopf algebra $k[X]$, with the co-product given by $\Delta(X)=X\otimes 1+1\otimes X$.
\item The {\bf multiplicative group} $\mathbb{G}_m$ is the functor $R\rightarrow R^{\times}$ (multiplicative group),  $R\in\mathbf{cAlg}_k$, represented by the Hopf algebra $k[X,Y]/(XY-1)\cong k[X,X^{-1}]$, with the co-product given by $\Delta(X)=X\otimes X$.
\end{enumerate}

For $(G,\mu,\iota)$ an algebraic group with identity element $e_G$, let
$A:= k[G]$ be its {\bf co-ordinate algebra}. A compatible co-algebra structure is then induced as follows: 
\begin{itemize}
\item the {\bf co-unit} or {\bf augmentation} is given by  
$$
\varepsilon: k[G]\rightarrow k,\quad f\mapsto f(e_G),
$$
with $k\cong k[e]$ being the affine variety corresponding to the group of one element $e$,\\
\item the {\bf co-product} by
\begin{eqnarray*}
\Delta&:=&\mu^*:k[G]\rightarrow k[G]\otimes_k k[G], \\
f&\mapsto&\sum g_i\otimes h_i\quad\text{if}\quad f(xy)=\sum g_i(x)h_i(y),\quad\text{for $x,y\in G$},
\end{eqnarray*}
\item the {\bf co-inverse} or {\bf antipode} by
$$
S:=\iota^*:k[G]\rightarrow k[G],\quad (\iota^*f)(x)=(f\circ\iota)(x)=f(x^{-1})\quad\text{for all $x\in G$}.
$$
\end{itemize}
It follows from the group axioms that this defines the structure of a {\bf Hopf algebra} on $k[G]$, and conversely, a Hopf algebra structure on $k[G]$ defines the structure of an algebraic group on the variety $G$.

\begin{df}
Let $G$ be an affine group scheme. A homomorphism $\chi:G\rightarrow\mathbb{G}_m$, is called a {\bf character}.
\begin{df} For $H$ a Hopf algebra, the kernel of the co-unit $\ker(\varepsilon)$,  is called the {\bf augmentation ideal} of $H$.
\end{df}

\end{df}
Generally, the following objects are in bijective correspondence:
\begin{center}
\begin{tabular}{lcl} Group $G$ & &Hopf algebra $H$\\ \hline Closed subgroups  & and &  Hopf ideals  \\Normal subgroups & and &  kernels of morphisms \\Characters & and  & group-like elements  \end{tabular}
\end{center}

\begin{df}
A {\bf pro-finite $k$-algebra} $A$, is a topological $k$-algebra which satisfies: 
$A$ is decreasingly filtered by a family of ideals $\{J_i\}_{i\in\N}$, i.e. $A\supset J_i\supset J_{i+1}$, $i\in\N$, such that 
$A/J_i$ is both a {\em discrete} topological space and a {\em finite} $k$-algebra, and the isomorphism 
$$
A\cong\varprojlim_{i\in\N} A/J_i,
$$
holds for the inverse limit. 
\end{df}
\begin{df}
An {\bf affine formal scheme}, is a $k$-functor $X$, which is representable by a pro-finite $k$-algebra $A$, and it is denoted as $X={\specf}_k(A)$. 
\end{df}
\begin{df}
A group object in the category of formal schemes is called a {\bf formal group}.
A formal group $G=\specf_k(A)$ is {\bf smooth} if $A$ is isomorphic to a power series algebra $k[[X_1,\dots, X_n]]$. 
\end{df}

The co-multiplication in a smooth formal group is described by a family of power series. Namely, the coproduct $\Delta:A\rightarrow A\hat{\otimes}_k A$, is given by a set of $n$ formal power series in $2n$ variables, $F(\underline{X},\underline{Y})=(F_1(\underline{X},\underline{Y}),\dots,F_n(\underline{X},\underline{Y}))$, with $\underline{X} := (X_1,\dots,X_n)$, $\underline{Y} := (Y_1,\dots,Y_n)$, and $A\hat{\otimes}_k A\cong k[[\underline{X},\underline{Y}]]$, where $\hat{\otimes}$ denotes the {\em completed} tensor product. In fact, this defines the structure of an {\bf $n$-dimensional formal group law}. 

More generally, let $I$ be an arbitrary set and $\underline{X}=(X_i)_{i\in I}$ a family of {\em commuting} indeterminantes $X_i$, indexed by $I$. Denote by $\mathbf{n}:I\rightarrow\N$ (or equivalently $\mathbf{m}$ etc.) any non-negative integer-valued function with {\bf finite support} $\supp(\mathbf{n})$, i.e. with $\mathbf{n}(i)=0$ for all but finitely many $i\in I$. 

Under the above assumptions, the following quantities are well-defined:
\begin{itemize}
\item $|\mathbf{n}|:=\sum_{i\in I}\mathbf{n}(i)$, the norm of the vector $\mathbf{n}$,
\item $X^{\mathbf{n}}:=\prod_{i\in\supp(\mathbf{n})} X_i^{\mathbf{n}(i)}$, a monomial,
\item the {\bf degree} of a monomial $X^{\mathbf{n}}$ is $\deg(X^{\mathbf{n}}):=|\mathbf{n}|$.
\end{itemize}
The ring of {\bf commuting formal power series} $R[[X_i|i\in I]]=R[[\underline{X}]]$, $R\in\mathbf{cRing}$, is the set of all (infinite) formal sums with $R$-valued coefficients $c_{\mathbf{n}}$ of the form
$$
\sum_{\mathbf{n}} c_{\mathbf{n}}X^{\mathbf{n}},\qquad c_{\mathbf{n}} \in R.
$$
$R[[X_i|i\in I]]$ is a {\em local} ring with the unique maximal ideal $\mathfrak{m}$ given by the set of power series with zero constant term, i.e. $c_{\mathbf{n}}=0$, for $|\mathbf{n}|=0$. Considered as an $R$-module, it has the following direct sum decomposition:
$$
R[[X_i|i\in I]]=R\oplus\mathfrak{m}.
$$
Therefore, every smooth formal group is connected as $A\cong k[[X_1,\dots, X_n]]$ is an $n$-dimensional local ring.

\begin{df}
\label{FG}
A {\bf  formal group law} over $R\in\mathbf{cRing}$, indexed by an (infinite) set $I$, is an $|I|$-tuple of formal power series $F=(F_i)_{i\in I}$, $F_i(\underline{X},\underline{Y})\in R[[X_i, Y_i~|~i\in I]]$, $i\in I$, with
$$
F_i(\underline{X},\underline{Y})=\sum_{\mathbf{m},\mathbf{n}} c_{\mathbf{m},\mathbf{n}}(i)\underline{X}^{\mathbf{m}} \underline{Y}^{\mathbf{n}},\quad\qquad c_{\mathbf{m},\mathbf{n}}(i)\in R,
$$
such that for all $i\in I$, the following properties are satisfied:
\begin{enumerate}
\item $F_i(\underline{X},\underline{0})=X_i$ and $F_i(\underline{0},\underline{Y})=Y_i$.
\item For every $\mathbf{m},\mathbf{n}$, and for almost all $i\in I$, $c_{\mathbf{m},\mathbf{n}}(i)=0$. 
\item $F_i(F(\underline{X},\underline{Y}),\underline{Z}))=F_i(\underline{X},F(\underline{Y},\underline{Z}))$, for $i\in I$.
\item $F$ is {\bf commutative} if $F_i(\underline{X},\underline{Y})=F_i(\underline{Y},\underline{X})$ for all $i\in I$.
\end{enumerate}
\end{df}
The following remarks are in place:
The first property in Definition~(\ref{FG}) implies that every $F_i$ is of the form
$$
F_i(\underline{X},\underline{Y})=X_i+Y_i+\sum_{\begin{subarray}{l} |\mathbf{m}|\geq 1\\ |\mathbf{n}|\geq 1 \end{subarray}} c_{\mathbf{m},\mathbf{n}}(i)X^{\mathbf{m}}Y^{\mathbf{n}}\cong X_i+Y_i\quad \mod(\deg~2).
$$
Therefore, only mixed terms in $X_i$ and $Y_j$ appear in degree~$\geq2$. It follows from the first axiom that no constant term is involved, and so the substitution in 3. is well-defined.  

The first three axioms imply the existence of an {\bf inverse}, cf. e.g.~\cite{Froh,H2}. Namely, for a given $F$, there is a unique sequence $\iota=(\iota_{\ell})_{\ell\in I}$  of power series $\iota_{\ell}(\underline{X})$, such that $F_{\ell}(\underline{X},\iota(\underline{X}))=F_{\ell}(\iota(\underline{X}), \underline{X})=\underline{0}$.

Let us now discuss the relation between {\bf Lie groups} and {\bf Lie algebras} in our context. Instead of looking at the general picture, we shall restrict ourselves to the specificities of formal group laws.

For simplicity, let $I:=\N^{\times}$. As
\begin{equation}
\label{gr_law_bilinear_f}
F_i(\underline{X},\underline{Y})\equiv X_i+Y_i+\underbrace{\sum_{k,l\in\N^{\times}} c^i_{kl} X_lY_k}_{=:B_i(\underline{X},\underline{Y})}\quad\mod (\deg~3),
\end{equation}
the sum $B_i(\underline{X},\underline{Y})$ defines a {\bf bilinear form}, i.e. a sum of quadratic monomials in $X_1,X_2,X_3,\dots$ and $Y_1,Y_2,Y_3,\dots$.

A Lie algebra structure can be defined on $R^{\N^{\times}}$, with the bracket 
$
[~,~]:R^{\N^{\times}}\times R^{\N^{\times}}\rightarrow R^{\N^{\times}},
$
be given by
\begin{equation}
\label{FG_Lie}
[\underline{X},\underline{Y}]:=B(\underline{X},\underline{Y})-B(\underline{Y},\underline{X}),
\end{equation} 
where $B(\underline{X},\underline{Y}):=(B_1(\underline{X},\underline{Y}),B_2(\underline{X},\underline{Y}),\dots, B_n(\underline{X},\underline{Y}),\dots)$ is the vector of bilinear forms with the $i$th component being equal to $B_i$.
In terms of the canonical basis vectors $e_j:=(0,\dots,0,1,0\dots)$, it is given by 
\begin{equation}
[e_j,e_k]=\sum_{i=1}^{\infty} (c^i_{jk}-c^i_{kj}) e_i=(c^1_{jk}-c^1_{kj},c^2_{jk}-c^2_{kj},c^3_{jk}-c^3_{kj},\dots)~.
\end{equation}
$B(\underline{X},\underline{Y})$ is {\em symmetric}, i.e. $B(\underline{X},\underline{Y})=B(\underline{Y},\underline{X})$,   iff the corresponding Lie algebra is abelian.
\begin{df}
The {\bf Lie algebra $L(F)$ of the formal group law}  $F(\underline{X},\underline{Y})$,  is the Lie algebra defined in~(\ref{FG_Lie}). 
\end{df}

Let us recall the main theorems of {\bf Formal Lie Theory} in a condensed form.
\begin{thm}[cf. e.g.~\cite{Froh,H1}]
\label{Q-theorems}
Let $R$ be a $\Q$-algebra, and $F,G$ formal group laws over $R$. 
\begin{enumerate}
\item The formal groups $F$ and $G$ are isomorphic if and only if the Lie algebras $L(F)$ and $L(G)$ are isomorphic, i.e.
$
F\cong G\Leftrightarrow L(F)\cong L(G).
$
\item Let $F$ be a {\bf commutative} formal group law. Then 
\begin{enumerate}
\item $L(F)=0$, i.e. the Lie algebra is trivial, and
\item $F$ is isomorphic to the {\bf additive group law} $\hat{\mathbb{G}}_a$, of dimension $\dim_R(F)$, with
$$
(\hat{\mathbb{G}}_a)_i(\underline{X},\underline{Y}):=X_i+Y_i
\qquad\text{for $i\in\{1,\dots,\dim_R(F)\}$}.
$$
\end{enumerate}
\end{enumerate}
\end{thm}
\subsection{The Hopf algebras $\mathbf{Symm}$ and $\Aut({\OO})$}
\label{examples}
Let us discuss two Hopf algebras, which are central to several mathematical objects, cf.~e.g.~\cite{FG-B,H1,H2}. Their connections with Free Probability were discussed in~\cite{FMcK1,FMcK2,MN1,MN2}.

Let $\Lambda:\mathbf{cRing}\rightarrow \mathbf{Grp}$ be the functor:
\begin{equation}
\label{lambda_ring}
\Lambda(R) := (1 + zR[[z]])^{\times}=\{f(z)=1+a_1z+a_2z^2+\dots |~a_j\in R\}. 
\end{equation}
The multiplication of the power series induces the structure of an abelian group, with $f(z)\equiv 1$, as additive neutral element.
Any ring homomorphism $f:R\rightarrow R'$ induces a morphism $\Lambda(f)$ of Abelian groups as follows:
$$
\Lambda(f)(1+a_1t+a_2t^2+a_3t^3+\dots):=1+f(a_1)t+f(a_2)t^2+f(a_3)t^3+\dots.
$$

The group-valued functor~(\ref{lambda_ring}) is known~\cite{H1,H2} to be represented by the $\Z$-algebra 
\begin{equation}
\label{Symm}
\mathbf{Symm}:=\Z[h_{*}]:=\Z[h_1,h_2,h_3,\dots]~,
\end{equation}
i.e. the polynomial ring in infinitely many commuting variables $h_n$, $n\in\N^{\times}$, with integer coefficients. Indeed, any homomorphism of unital rings $\varphi:\mathbf{Symm}\rightarrow R$, is uniquely determined by the values $h_n\mapsto\varphi(h_n)$, which in turn determines a power series 
$$
\sum_{n=0}^{\infty}\varphi(h_n) z^n\in\Lambda(R)~,
$$
where we used the convention $h_0\equiv1$, and hence $\varphi(h_0)=1$.

For $k\in\mathbf{cRing}$, let 
$
\mathbf{Symm}_k:=\Z[h_{*}]\otimes_{\Z} k\cong k[h_1,h_2,h_3,\dots].
$

The co-ordinate Hopf algebra structure on $\mathbf{Symm}$ is obtained from the group structure of $\Lambda$ as follows. For an integer $n\geq1$, we have
\begin{itemize}
\item (unit) $e: 1 \mapsto  h_0\equiv1$,
\item (multiplication) $\mu: h_i\otimes h_j \mapsto  h_i\cdot h_j$,
\item (co-unit) $\varepsilon:  h_n \mapsto 0$,
\item (co-multiplication) $\Delta:  h_n \mapsto \sum_{j=0}^{n} h_j\otimes h_{n-j}$,
\item (co-inverse) $S : h_n  \mapsto  -\sum_{j=0}^{n-1}S(h_j)h_{n-j}$ with $S(h_0)=1$. 
\end{itemize}

If we assign to $h_i$ weight $i$, then $\mathbf{Symm}$ is a {\bf $\N$-graded} and {\bf connected} Hopf algebra, i.e. 
\begin{equation*}
\mathbf{Symm}=\bigoplus_{n=0}^{\infty}\mathbf{Symm}_n,
\end{equation*}
which is a direct sum of homogeneous abelian subgroups, satisfying $\mathbf{Symm}_0\cong \Z$, and with $\mathbf{Symm}_n$ being the span of the monomials in the $h_i$'s of total weight $n$.
It is {\bf connected} because the degree $0$ part  is of rank one.

The multiplication, co-multiplication and the antipode are compatible with the grading, i.e. we require
\begin{eqnarray*}
\mu(\mathbf{Symm}_m\otimes\mathbf{Symm}_n) & \subseteq & \mathbf{Symm}_{m+n}, \\
\Delta(\mathbf{Symm}_n) & \subseteq & \bigoplus^n_{j=0} \mathbf{Symm}_j\otimes\mathbf{Symm}_{n-j},\\
S(\mathbf{Symm}_n)&\subseteq& \mathbf{Symm}_n.
\end{eqnarray*}
As $\Lambda$  takes values in formal groups, it is the inverse limit of regular schemes, i.e.  
$$
\mathbf{Symm}=\mathop{\lim_{\longleftarrow}}_{n\in\N}\Z[h_1,\dots,h_n].
$$
The pair 
$$
({\Z[h_*]}, 1+h_1t +h_2t^2+h_3t^3 +\dots)
$$
represents $\Lambda$, with the natural transformation 
given by:
$$
\Hom_{\mathbf{cRing}}(\Z[h^*],R)\ni f\mapsto\Lambda(f)(H(t))=1+f(h_1)t+f(t_2)t^2+f(h_3)t^3+\dots\in\Lambda(R),
$$
for $R\in\mathbf{cRing}$, and 
$
H(t):=1+h_1t+h_2t^2+h_3t^3+\dots\in\Lambda(\Z[h_*]).
$
The latter series is the generating function for the {\bf power-sum symmetric functions} (more precisely its ghost components).

The second example we shall present, is the {\bf Faà di Bruno Hopf algebra}, cf. e.g.~\cite{FG-B}, and also~\cite{FMcK2} and references therein. For $A\in\mathbf{cAlg}_k$, let $A[[z]]=:{\OO}_A$ be the complete topological $k$-algebra of formal power series in an indeterminate $z$. For a continuous $k$-algebra homomorphism $\varphi$, we have
$$
\varphi\left(\sum_{n=0}^{\infty}a_nz^n\right)=\sum_{n=0}^{\infty}a_n\varphi(z^n)=\sum_{n=0}^{\infty}a_n\varphi(z)^n~,
$$
and it is completely determined by its value on the generator $z$. The set of automorphisms of $A[[z]]$, which is an infinite dimensional Lie group, has to be of the form 
$$
\Aut({\OO}_A):=\{\varphi(z)=a_1z+a_2z^2+a_3z^3+\cdots~|a_1\in A^{\times}, a_n\in A, n\geq2\},
$$
with the subgroup
$$
\Aut_+({\OO}_A):=\{\varphi(z)=z+a_2z^2+a_3z^3+\cdots~\},
$$
consisting of automorphisms which have a unit tangent vector at the identity. The group operation is given by the non-commutative composition of power series, i.e. for power series $f(z)=z+a_2z^2+a_3z^3+\dots$ and $g(z)=z+b_2z^2+b_3z^3+\dots$, the composite is $(g\circ f)(z):=g(f(z))$. The induced Hopf algebra structure on the co-ordinate ring $k[X_1,X_2,X_3,\dots]$ is also called the {\bf Faà di Bruno Hopf algebra}, although it appeared earlier in complex cobordism, cf.~the remarks in~\cite{FMcK1}. It is a graded connected Hopf algebra and it defines a pro-unipotent affine group scheme.

Its co-inverse or antipode is given in terms of the co-ordinate functions $X_n(f):=a_n$, $n\in\N^{\times}$, by
\begin{equation}
\label{FdB_coniverse}
(S_{\operatorname{FdB}}X_n)(f):=X_n(f^{-1}),
\end{equation}
with $f^{-1}$ denoting the inverse with respect to composition. It is known, cf.~\cite{FG-B}, that it can be calculated by the {\bf Lagrange-Bührmann inversion formula}, or in closed form by the (generalised) {\bf Zimmermann forest formula}.

\section{Non-commutative $k$-probability spaces}
Here we generalise basic notions from free probability to arbitrary commutative rings. Classical references are~\cite{V,NS}.

Let $k$ be a commutative ring with unit and $\mathcal{A}$ an associative but not necessarily commutative unital $k$-algebra.

For two algebras ${A},{B}\in\mathbf{Alg}_k$, let $\Hom_k({A},{B})$ denote the set of $k$-linear maps and $\Hom_{k,1}(A,B)$ the subset of $k$-linear maps with the additional property that $1_A\mapsto 1_B$, which is called {\bf unital}. We note that $\Hom_{k,1}(A,B)$ is an {\em affine subspace} of $\Hom_{k}(A,B)$.

\begin{df}
A {\bf non-commutative $R$-valued $k$-probability space} consists of  a pair $(\mathcal{A},\phi)$, with $\mathcal{A}\in\mathbf{Alg}_k$ and a {\em fixed} $k$-linear functional $\phi$ with values in $R\in\mathbf{cAlg}_k$, such that  
$
\phi(1_{\mathcal{A}})=1_R.
$\\
The elements $a\in \mathcal{A}$ are called  {\bf non-commutative random variables}. 
\end{df}

\begin{rem} Equally, we shall speak about a {\bf non-commutative $R-k$-probability space}, and so the usual theory corresponds to a non-commutative $\C-\C$-probability space. 
\end{rem}

For a vector $\underline{a}=(a_1,\dots, a_s)\in\mathcal{A}^s$,  $s\in\N^{\times}$,  and a word $w=(i_1\dots i_n)\in[s]^*$, let
$\underline{a}_w:=a_{i_1}\cdots a_{i_n}$ and $\underline{a}_{\emptyset}:=1_{\mathcal{A}}$.
\begin{df}
The {\bf $s$-dimensional distribution} or {\bf law} of $\phi$, with values in $R$, is the map
$$
\mu:\mathcal{A}^s\rightarrow \Hom_{k,1}(k\langle z_1,\dots, z_s\rangle,R),
$$
which assigns to every $s$-tuple $\underline{a}:=(a_1,\dots, a_s)$ a $k$-linear functional such that 
\begin{eqnarray}
\label{s-distribution}
\mu_{(a_1,\dots, a_s)}&:&k\langle z_1,\dots, z_s\rangle\rightarrow R,\\\nonumber
& & z_w\mapsto \mu_{(a_1,\dots, a_s)}(\underline{a}_w):=\phi(\underline{a}_w)\qquad  \forall w\in[s]^*~.
\end{eqnarray}
$\mu_{(a_1,\dots, a_n)}$ is called the {\bf $R$-valued distribution} of the $s$-tuple $(a_1,\dots, a_s)$. 
\end{df}

\begin{df}
The  {\bf joint moments of order $|w|$ of $\underline{a}$}, are the values
\begin{equation}
\label{moment}
m_w(\underline{a}_w):=\phi(\underline{a}_w)=\mu_{(a_1,\dots, a_s)}(\underline{a}_w)\in R.
\end{equation}
The {\bf $\mathcal{M}$-transform} gives the {\bf joint moment series}
\begin{equation}
\label{moment-map}
\mathcal{M}:\mathcal{A}^s\rightarrow R_+\langle\langle z_1,\dots,z_s\rangle\rangle,
\end{equation}
\end{df}
defined as
\begin{equation}
\label{moment_series}
\mathcal{M}_{(a_1,\dots, a_s)}(z_1,\dots, z_s):=\sum_{\begin{subarray}{l} w\in[s]^*\\ |w|\geq 1 \end{subarray}} m_w(a_w) z_w.
\end{equation}

\begin{prop}
Let $(\mathcal{A},\phi)$ be a non-commutative $R-k$-probability space.  Any $\varphi\in\Hom_{\mathbf{cAlg}_k}(R,R')$ induces an $R'$-valued non-commutative $k$-probability space, with the unital linear functional $\phi'$ given as the push-forward by $\varphi$, as in
\[
\begin{xy}
  \xymatrix{
   \mathcal{A}\ar[r]^{\phi}\ar[rd]_{\phi'=\varphi_*\phi}    & R\ar[d]^{\varphi}\\
    & R'
               }
\end{xy}
\]
\end{prop}
\begin{exmpl}
\end{exmpl}
Let $(X,{\OO}_X)$ be a sheaf of $k$-algebras. For $\mathcal{A}\in\mathbf{Alg}_k$ define any unital $k$-linear functional 
$$
\phi:=\phi_X:\mathcal{A}\rightarrow{\OO}_X(X).
$$
This induces a family of non-commutative $({\OO}_X(U)-k)$-probability spaces, indexed by the open sets $U\subset X$, with
$
\phi_U:\mathcal{A}\rightarrow{\OO}(U)
$
given by 
$
\phi_U=\rho^X_U\circ\phi_X.
$
\section{Free $k$-cumulants}
References for this section are~\cite{Krew,N,NS,S97}. 
\subsection{Non-crossing partitions}
\label{noncross-sec}
\begin{df}
Let $S$ be a finite set. A {\bf partition} $\pi$ of $S$ is a collection $\pi:=\{V_1,\dots, V_r\}$ of non-empty, pairwise disjoint subsets $V_i\subset S$, $ i=1,\dots,r$, called the {\bf blocks} of $\pi$, with union $S$, i.e. $S=\coprod_{i=1}^r V_i$. 
\end{df}
The number of blocks in a given partition $\pi$ is denote by $|\pi|$. We shall write $\operatorname{P}(S)$ for the set of all partitions of $S$, and for $S=\{1,\dots,n\}$ we use $\operatorname{P}(n)$. 

On the set $[n]:=\{1,\dots,n\}$, we define an equivalence relation $\sim_{\pi}$ as follows:  for $1\leq p,q\leq n$, let
$$
p\sim_{\pi} q:\quad\Leftrightarrow\exists i: p,q\in V_i, \quad\text{i.e. $p,q$ are in the same block $V_i$ of $\pi$.}
$$
\begin{df}
A partition $\pi$ of $S=\{1,\dots,n\}$ is called {\bf crossing} if there exist $p_1<q_1<p_2<q_2$ in $S$ such that $p_1\sim_{\pi} p_2$, $q_1\sim_{\pi} q_2$ but $p_2\nsim_{\pi} q_1$. If $\pi$ is not crossing then is called {\bf non-crossing}. The subset of all non-crossing partitions of $[n]$ is denoted by $\operatorname{NC}(n)$.
\end{df}
The set $\operatorname{NC}(n)$ is partially ordered, i.e. a {\bf poset}, by the {\bf reversed refinement order}, which gives it also a lattice structure. For two partitions $\pi,\sigma\in\operatorname{NC}(n)$, we write $\sigma\leq\pi$, if every block of $\sigma$ is contained in some block of $\pi$. The {\bf maximal element} is given by the partition
$$
1_n=\{(1,\dots, n)\},
$$
and the {\bf minimal element} by
$$
0_n=\{(1),\dots,(n)\}.
$$
\begin{df}
Let $(S,\leq)$ be a partially ordered finite set and $p,q\in S$. The element
$$
p\vee q:=\min\{t\in S~|~\text{$p\leq t$ and $q\leq t$}\}
$$
is called the {\bf join} of $p$ and $q$, if it exists. 
\end{df}

The notion of non-crossing partition originated with {\bf Kreweras}~\cite{Krew}, and it can be generalised to any finite, {totally ordered} set. He also gave the following construction, which has no counterpart in the general case.

Introduce an additional ordered set $[\overline{n}]:=\{\overline{1},\dots,\overline{n}\}$, and denote by $\operatorname{NC}(\overline{n})$ all non-crossing partitions of $[\overline{n}]$. Equally, define the ordered set $[n,\overline{n}]:=\{1,\overline{1},\dots,n,\overline{n}\}$ and let $\operatorname{NC}(n,\overline{n})$ be the set of all non-crossing partitions of $[n,\overline{n}]$. 
\begin{df}
Let $\pi\in\operatorname{NC}(n)$ be a non-crossing partition. The {\bf Kreweras complement} $K(\pi)$  of $\pi$, is the maximal non-crossing partition $K(\pi)\in\operatorname{NC}(\overline{n})$ such that $\pi\cup K(\pi)\in\operatorname{NC}(n,\overline{n})$.
\end{df}
By identifying $\operatorname{NC}(\overline{n})$ with $\operatorname{NC}({n})$, the Kreweras complement $K(-)$ induces a bijection
$$
K:\operatorname{NC}({n})\rightarrow \operatorname{NC}({n}), 
$$
which is has no fixed points for $n\geq2$, i.e. $\pi\neq K(\pi)$ for all $\pi\in\operatorname{NC}(n)$.  For every integer $n\geq1$, it satisfies the  relation~\cite{Krew}: 
\begin{equation}
\label{part-Krew_rel}
|\pi| + |K(\pi)|=n+1\qquad\text{for all $\pi\in\operatorname{NC}(n)$}.
\end{equation}
\begin{rem}
\label{pi-K^2(pi)}
For $\pi\in\operatorname{NC}(n)$, the partition $K^2(\pi)$ has the same block structure as $\pi$ and it can be obtained from $\pi$ by a global cyclic permutation.
\end{rem}
\begin{exmpl}
\end{exmpl}
We list the non-crossing partitions for $n=1,\dots,4$, with the subscript denoting an enumeration,  and the corresponding Kreweras complement, indicated by an arrow, e.g. $4\rightarrow 9$, means that the Kreweras complement of the partition number 4 is the partition number 9. 
\begin{description}
\item[$n=1$]
\begin{eqnarray*}
(1)_1
 \end{eqnarray*}
 \begin{eqnarray*}
1\rightarrow 1
\end{eqnarray*}
\item[$n=2$]
\begin{eqnarray*}
(12)_1 \\
(1|2)_{2}
 \end{eqnarray*}
  \begin{eqnarray*}
1\rightarrow 2, 2\rightarrow 1
\end{eqnarray*}
\item[$n=3$]
$$
\begin{array}{ccc} & (123)_1 &  \\(1|23)_2 & (12|3)_3 &  (13|2)_4 \\ & (1|2|3|)_5 & \end{array}
$$
\begin{eqnarray*}
1\rightarrow 5, 2\rightarrow 4,  3\rightarrow 2, 4\rightarrow 3, 5\rightarrow 1
\end{eqnarray*}
\item[$n=4$]
$$
\begin{array}{cccccc} &  & (1234)_1 &  &  &  \\(1|234)_2 & (134|2)_3 & (124|3)_4 & (123|4)_5 & (12|34)_6 & (14|23)_7 \\(1|2|34)_8 & (1|23|4)_9 & (12|3|4)_{10} & (1|24|3)_{11} & (14|2|3)_{12} & (13|2|4)_{13} \\ &  & (1|2|3|4)_{14} &  &  & \end{array}
$$
\begin{eqnarray*}
1\rightarrow 14, 2\rightarrow 12,  3\rightarrow 10, 4\rightarrow 9, 5\rightarrow 8, 6\rightarrow 11, 7\rightarrow 13 \\
8\rightarrow 4, 9\rightarrow 3,  10\rightarrow 2, 11\rightarrow 7, 12\rightarrow 5, 13\rightarrow 6, 14\rightarrow 1
\end{eqnarray*}
\end{description}

\subsection{Free cumulants and combinatorial freeness}
Let $(\mathcal{A},\phi)$ be a non-commutative $R-k$-probability space. 
\begin{df}
The $R$-valued {\bf free cumulants} $\kappa_n$, are a family of $k$-multilinear functionals $\kappa_n:\mathcal{A}^n\rightarrow R$, $n\in\N$, which are {\em recursively} defined by the {\bf moment-cumulant formula}:
\begin{equation}
\label{free-cumulants}
\kappa_1(a_1):=\phi(a_1),\qquad{and}\qquad \phi(a_1\cdots a_n)=\sum_{\pi\in\operatorname{NC}(n)}\kappa_{\pi}[a_1,\dots,a_n],\quad
n\geq2
\end{equation}
where for $\pi=\{V_1,\dots, V_r\}\in\operatorname{NC}(n)$, we define
\begin{equation}
\label{free-cumulants2}
\kappa_{\pi}[a_1,\dots, a_n]:=\prod_{j=1}^r \kappa_{V_j}[a_1,\dots, a_n],
\end{equation}
using the notation
$$
k_V[a_1,\dots, a_n]:=k_{m}(a_{v_1},\dots,a_{v_m})
$$
for the restriction to a block $V=(v_1,\dots,v_m)\subset[n]$, with $m=|V|\leq n$.

\end{df}

If we define, by abuse of notation, $k_n(a):=k_n(a,\dots,a)$, then the resulting series $(k_n(a))_{n\in\N}$ is called the {\bf $R$-valued free cumulants series of $a$}, or just the {\bf $R$-free cumulants} of $a$. This can be generalised to several variables. 
\begin{df}
Let $\underline{a}=(a_1,\dots,a_s)\in\mathcal{A}^s$, $s\geq 1$. The {\bf $\mathcal{R}$-transform} with values in $R\in\mathbf{cAlg}_k$, is the map 
$
\mathcal{R}:\mathcal{A}^s\rightarrow R_+\langle\langle z_1,\dots, z_s\rangle\rangle
$
which assigns to every $\underline{a}$ the formal power series $\mathcal{R}(\underline{a})$ in $s$ non-commuting variables $\{z_1,\dots, z_s\}$, given by  
\begin{eqnarray}
\label{R-trafo}
\mathcal{R}(\underline{a}):=\mathcal{R}_{a_1,\dots,a_s}(z_1,\dots,z_s)&:=&\sum^{\infty}_{n=1}\sum^s_{i_1,\dots, i_n=1}\kappa_n(a_{i_1},\dots,a_{i_n})z_{i_1}\cdots z_{i_n}\\
\nonumber
&=&\sum_{|w|\geq1}\kappa_{|w|}(\underline{a}_w) z_w.
\end{eqnarray} 
\end{df}

It is important to note that formula~(\ref{free-cumulants}) leads to {\bf polynomial expressions with integer coefficients} when explicitly written out, cf.~\cite{NS}, p.~176. 
\begin{exmpl}
\end{exmpl}
Let us list the relations between the free cumulants and the moments for $n=1,2,3$:

\begin{description}
\item[$n=1$]: $\kappa_1(a_1)=\phi(a_1)$,
\item[$n=2$]: $\kappa_1(a_1,a_2)=\phi(a_1a_2)-\phi(a_1)\phi(a_2)$,
\item[$n=3$]:
$\kappa_1(a_1,a_2,a_3)=\phi(a_1a_2a_3)-\phi(a_1)\phi(a_2a_3)-\phi(a_1a_2)\phi(a_3)-\phi(a_1a_3)\phi(a_2)+2\phi(a_1)\phi(a_2)\phi(a_3)$.
\end{description}
As with the complex valued case~\cite{NS}, one proves the following
\begin{prop} Let $(\mathcal{A},\phi)$ be a non-commutative $R-k$-probability space. For $n\geq2$, and arbitrary $a_1,\dots, a_n\in\mathcal{A}$, the following holds: $\kappa_n(a_1,\dots, a_n)=0$ if there exists $1\leq i\leq n$ such that $a_i\in k 1_{\mathcal{A}}$.
\end{prop}

Basically, the free cumulants measure how much the fixed $k$-linear functional $\phi$ fails to be a $k$-algebra morphism, as was observed in~\cite{NS}.
\begin{df}
Let $(\mathcal{A},\phi)$ be an non-commutative $R-k$-probability space and $M_1,\dots, M_n\subset\mathcal{A}$. The subsets $M_1,\dots, M_n$ are said to have {\bf vanishing mixed cumulants} if 
$$
\kappa_m(a_1,\dots, a_m)=0,
$$
whenever $a_1\in M_{i_1},\dots,a_m\in M_{i_m}$ and there exist indices $1\leq s<t\leq m$, such that $i_s\neq i_t$.
\end{df}
Now we give a definition of freeness which is required in a general setting.  
\begin{df}
Let $(\mathcal{A},\phi)$ be a non-commutative $R-k$-probability space. The subsets $M_1,\dots, M_n\subset\mathcal{A}$ are called {\bf combinatorially free} if they have vanishing mixed cumulants.
\end{df}
An adapted version of Voiculescu's original definition of freeness is 
\begin{df}
\label{freeness}
Let $(\mathcal{A},\phi)$ be a non-commutative $k$-probability space with values in $k$. Further, let $I$ be an arbitrary non-empty index set and $(\mathcal{A}_i)_{i\in I}$ a collection of unital subalgebras, i.e. $1_\mathcal{A}\in \mathcal{A}_i\subset \mathcal{A}$, $i\in I$. 

The family $(\mathcal{A}_i)_{i\in I}$ is {\bf freely independent} with respect to $\phi$ if 
$$
\phi(a_1\cdots a_n)=0,
$$
given that
\begin{enumerate}
\item $a_j\in \mathcal{A}_{i_j}$ for $1\leq j\leq n$,
\item $a_j\in\ker(\phi)$ for $1\leq j\leq n$, and
\item $i_j\neq i_{j+1}$ for $1\leq j< n$, i.e. {\em consecutive} factors $a_j$ and $a_{j+1}$ are from different subalgebras.
\end{enumerate}
A collection of subsets $X_i\subset A$, or elements $a_i\in A$, $i\in I$, is called {\bf free} if the family of subalgebras $A_i$ generated by $\{1_A\}\cup X_i$, respectively $\{1_A, a_i\}$, is free.
\end{df}
Sometimes we shall refer to the above definition as {\em classical} freeness. 
\begin{lem}
Let $a\in\mathcal{A}$ be any element in a non-commutative $k-k$-probability space $(\mathcal{A},\phi)$. The pair consisting of $(1_{\mathcal{A}},a)$ or $(0_{\mathcal{A}},a)$ is free. 
\end{lem}
\begin{proof}
The $k$-algebra generated by $1_{\mathcal{A}}$ is $k\cdot 1_{\mathcal{A}}$ and the only element $x$ it contains such that $\phi(x)=0$, is $0_{\mathcal{A}}$. But  any product of elements in $\mathcal{A}$ containing $0_{\mathcal{A}}$ is zero, and therefore in the kernel of $\phi$.

The second case follows from the definition of freeness, as we consider the unital subalgebra generated by an element, and therefore $\{1_{\mathcal{A}}\}$ and $\{1_{\mathcal{A}}, 0_{\mathcal{A}}\}$ generate the same algebra, and so we are in the situation of the first one.
\end{proof}

The following is based on the polynomial nature of the moment-cumulant formula.
 \begin{prop}
Let $(\mathcal{A},\phi)$ be a non-commutative $R-k$-probability space and $M_1,\dots, M_n\subset\mathcal{A}$ combinatorially free subsets. For any $\varphi\in\Hom_{\mathbf{cAlg}_k}(R,R')$ the above subsets are also combinatorially free in the non-commutative $R'-k$ probability space $(\mathcal{A},\phi')$ with
$$
\phi':=\varphi_*\phi=\varphi\circ \phi\quad\text{and}\quad\kappa'_n=\varphi_*\kappa_n=\varphi\circ\kappa_n.
$$
\end{prop}
\begin{proof}
This follows from the definitions and the properties of a ring homomorphism.
\end{proof}

The relation between the two notions of freeness is clarified next. In a more general setting they are not identical as we cannot a priori centre the random variables, as was first observed in~\cite{FN,F}.
\begin{prop} We have 
\begin{enumerate}
\item Let $\mathcal{A}_1,\dots, \mathcal{A}_n\subset\mathcal{A}$ be {\em combinatorially free} unital subalgebras of an $R$-valued non-commutative $k$-probability space $(\mathcal{A},\phi)$. Then they are also {\em classically free}, i.e. for all $m\geq 1$
$$
\phi(a_{1}\cdots a_{m})=0,
$$
whenever $a_{1}\in\mathcal{A}_{i_1},\dots, a_{m}\in\mathcal{A}_{i_m}$ with $i_1\neq i_2,\dots ,i_{m-1}\neq i_m$ and $\phi(a_i)=0$.
\item For any $k$-valued non-commutative $k$-probability space, i.e. $\phi:\mathcal{A}\rightarrow k$, classical and combinatorial freeness are equivalent.
\end{enumerate}
\end{prop}
\begin{proof} In order to show 1., we can apply the argument from~\cite{NS}, as every unital ring is a $\Z$-module, as $\Z1_R$ is in the centre of $R$. 

Let $a_j\in\mathcal{A}_{i_j}$, for $j=1,\dots,m$, with subsequent indices being different, i.e. $i_1\neq i_2,\dots, i_{m-1}\neq i_m$ and $\phi(a_j)=0$ for every $j=1,\dots, m$.

In order to have freeness, we need to show that $\phi(a_1\cdots a_m)=0$. By formula~(\ref{free-cumulants}) we have $\phi(a_1\dots a_m)=\sum_{\pi\in\operatorname{NC}(m)}\kappa_{\pi}[a_1,\dots, a_m]$  but then every summand vanishes as every product  $\kappa_{\pi}[a_1,\dots, a_m]=\prod_{V\in\pi}\kappa_V[a_1,\dots, v_m]$ contains a factor equal to $0$. 

Namely, if a partition $\pi=(V_1,\dots, V_r)$ contains a block $V_i=(v_l)$ for a $i= 1,\dots, r$ and an $l= 1,\dots, m$, i.e. which consists of a single element, then $\kappa_1(a_l)=\phi(a_l)=0$ by assumption and the respective product~(\ref{free-cumulants2}) is zero.

If it contains no block consisting of a single element then there must be at least one block having a minimum of two consecutive elements, i.e. of the form $V_i=(\dots v_l,v_{l+1},\dots)$. But then by assuming freeness, $\mathcal{A}_{v_l}\neq \mathcal{A}_{v_{l+1}}$ and therefore the mixed cumulants have to vanish, again by assumption on combinatorial freeness.

To show 2., we can use the ``centering  trick", i.e. consider random variables of the form
$$
a-\underbrace{\phi(a)}_{\in k}1_{\mathcal{A}}
$$
which satisfy $\phi(a-\phi(a)1_{\mathcal{A}})=\phi(a)-\phi(a)1_k=0$ as $\phi$ is $k$-linear and unital. 
Therefore the proof in~[\cite{NS} Thm.~11.16, p.183] carries over.
\end{proof}

In summary, we have:
\begin{itemize}
\item For $\phi:\mathcal{A}\rightarrow k$: combinatorially free~$\Leftrightarrow$~free, as we can centre the random variables.
\item For $\phi:\mathcal{A}\rightarrow R$: combinatorial freeness~$\Rightarrow$~free, but not the other way round. 
\end{itemize}
The following is useful for generalisations.  
\begin{df}
Let $(\mathcal{A},\phi)$ be a non-commutative $R-k$-probability space and $M_1,\dots, M_n\subset\mathcal{A}$ (combinatorially) free subsets. Given $R'\in\mathbf{cAlg}_k$ and $\varphi\in\Hom_{\mathbf{cAlg}_k}(R',R)$, a {\bf (combinatorially) free lift} or {\bf lifting} of $\phi$, is a unital $k$-linear map $\phi':{\mathcal{A}}\rightarrow R'$ such that $M_1,\dots, M_n\subset\mathcal{A}$ remain combinatorially free with respect to $\phi$ and $\phi'\circ\varphi=\phi$, i.e. such that
\[
\begin{xy}
  \xymatrix{
    &{R'}\ar[d]^{\varphi}\\           \mathcal{A}\ar[r]^{\phi}\ar[ur]^{\phi'} & R
               }
\end{xy}
\]
commutes.
\end{df}

\section{Addition and multiplication of combinatorially free random variables}
In this section we shall not only give a generalisation of the theory established in~\cite{N,NS,S97} but also extend it further, including the moment picture. Also, where suitable, we shall adapt some of the original proofs of~\cite{NS} to the present situation.   

The $n$-fold direct product $\mathcal{A}^n=
\prod_{i=1}^n \mathcal{A}
$ of a $k$-algebra $\mathcal{A}\in\mathbf{Alg}_k$ with itself is, with component-wise operations, an associative but not necessarily commutative $k$-algebra. For $\underline{a}=(a_1,\dots,a_n), \underline{b}=(b_1,\dots,b_n)\in \mathcal{A}^n$, and $\lambda\in k$, we have
\begin{eqnarray*}
\underline{a}+\underline{b} & := & (a_1+b_1,\dots,a_n+b_n), \\
\underline{a}\star\underline{b} & :=& (a_1\cdot b_1,\dots,a_n\cdot b_n),\\
\lambda\cdot\underline{a} & :=& (\lambda\cdot a_1,\dots,\lambda\cdot a_n).
\end{eqnarray*}
On the other hand, the set $k\langle\langle z_1,\dots, z_n\rangle\rangle$ is a $k$-module, and besides the usual non-commutative multiplicative structure~(\ref{noncom_prod}), one might also endow it with a commutative monoid structure given by term-wise multiplication, also called {\bf Hadamard} multiplication. 

In general, the moment map~(\ref{moment-map}) does not preserve any of the above algebraic structures, e.g. in the additive case we would have
$$
\mathcal{M}(\underline{a}+\underline{b})\neq\mathcal{M}(\underline{a})+\mathcal{M}(\underline{b}),
$$
and in the multiplicative case
$$
\mathcal{M}(\underline{a}\star\underline{b})\neq\mathcal{M}(\underline{a})\cdot\mathcal{M}(\underline{b}).
$$
But, for a non-commutative $R-k$-probability space $(\mathcal{A},\phi)$ the following holds. 
\begin{prop}
Let $M_1,M_2\subset\mathcal{A}$ be free subsets. For every $n\geq1$, $n$-tuples $\underline{a}=(a_1,\dots,a_n)\in M^n_1$ and $\underline{b}=(b_1,\dots,b_n)\in M^n_2$, we have
\begin{equation}
\kappa_n(\underline{a}+\underline{b}) =  \kappa_n(\underline{a})+_R\kappa_n(\underline{b}).
\end{equation}
\end{prop}
\begin{proof}
If we expand the expressions $\kappa_n(a_1+b_1,\dots, a_n+b_n)$ by using the multi-linearity, we get a sum of $2^n$ terms. These can be bijectively mapped onto a binary tree of height $n$ which represents words of length $n$ over the alphabet $\{a,b\}$. Exactly two of these words consist entirely of $a$'s or $b$'s and all the others contain both letters.  These correspond to mixed cumulants and have therefore to vanish according to our assumptions. So, the remaining sum is precisely $\kappa_n(a_1,\dots, a_n)+\kappa_n(b_1,\dots, b_n)$, which proves the claim.
 \end{proof}
\begin{prop}
\label{mult_free_conv}
Let $M_1,M_2\subset\mathcal{A}$ be {free} subsets. For every $n\geq1$, $n$-tuples $\underline{a}=(a_1,\dots,a_n)\in M^n_1$ and $\underline{b}=(b_1,\dots,b_n)\in M^n_2$, we have
\begin{equation}
\label{kumulant_mult}
\kappa_n(a_1b_1,\dots, a_n b_n)  =\sum_{\pi\in\operatorname{NC}(n)}  \kappa_{\pi}[\underline{a}]\cdot_R\kappa_{K(\pi)}[\underline{b}],
\end{equation}
where $K(\pi)$ denotes the Kreweras complement of $\pi$.
\end{prop}
In order to prove the above, we need
\begin{lem}
\label{1112}
Let $(\mathcal A,\phi)$ be a non-commutative $k$-probability space and $a_1,\dots,a_n\in\mathcal{A}$. For $m\leq n$, let $1\leq j_1<j_2<\dots< j_m=n$, and consider the partition $\sigma\in\operatorname{NC}(n)$, defined as:
$$
\sigma:=\left\{\{1,\dots,j_1\},\{j_1+1,\dots,j_2\},\dots,\{j_{m-1}+1,\dots,j_m\}\right\}.
$$
Then the following product formula holds:
$$
\kappa_n(a_1\cdots a_{j_1}, a_{j_1+1}\cdots a_{j_2},\dots, a_{j_{m-1}+1}\cdots a_{j_m})=\sum_{\substack{\pi\in\operatorname{NC}(n)\\ \pi\vee\sigma=1_{n}}}\kappa_{\pi}[a_1,\dots, a_n].
$$
\end{lem}
\begin{proof}
Cf.~\cite{NS}, p.180.
\end{proof}
\begin{proof}[Proof of Proposition~\ref{mult_free_conv}]
It follows from Lemma~\ref{1112} that the left side of equation~(\ref{kumulant_mult}) is equal to 
$$
\sum_{\substack{\pi\in\operatorname{NC}(2n)\\ \pi\vee\sigma=1_{2n}}}\kappa_{\pi}[a_1,b_1,a_2,\dots,b_{n-1},a_n,b_n],
$$
where $\sigma$ is the partition of the set $[2n]$ into pairs of successive elements, i.e. $\sigma:=\{(1,2),\dots, (2n-1,2n)\}$ and $\vee$ is the join. As mixed cumulants vanish, only those terms contribute which are indexed by non-crossing partitions $\pi$, which are a union of a non-crossing partition $\pi_a$ of $\{a_1,a_2,\dots,a_n\}$ and a non-crossing partition $\pi_b$ of $\{b_1,b_2,\dots,b_n\}$. The condition $\pi\vee\sigma=1_{2n}$ for such a partition $\pi$ is equivalent to $\pi_b=K(\pi_a)$, in which case the above sum becomes exactly the sum in equation~(\ref{kumulant_mult}).
\end{proof}
The following statements involve the operation of the ``boxed convolution $\boxtimes$", which is introduced in Section~\ref{boxconv_sec}, Definition~\ref{boxconv_df}. The essence is that for certain subsets, namely the (combinatorially) free ones, we get maps respecting the given algebraic structures.

\begin{prop}
\label{R-trafo-group-morph}
Let $(\mathcal{A},\phi)$ be a non-commutative $R-k$-probability space. The 
$\mathcal{R}$-transform is a morphism
$$
\mathcal{R}:\mathcal{A}^s\rightarrow R_+\langle\langle z_1,\dots, z_s\rangle\rangle,
$$
for those $\underline{a}=(a_1,\dots,a_s), \underline{b}=(b_1,\dots,b_s)\in\mathcal{A}^s$, for which the sets $\{a_1,\dots,a_s\}$ and $\{b_1,\dots,b_s\}$ are combinatorially free, i.e. we have
\begin{itemize}
\item in the {\bf additive} case 
\begin{equation}
\mathcal{R}(\underline{a}+\underline{b})=\mathcal{R}(\underline{a})+_R\mathcal{R}(\underline{b}),
\end{equation}
\item in the {\bf multiplicative} case
\begin{equation}
\label{R_trafo_mult}
\mathcal{R}(\underline{a}\star\underline{b})=\mathcal{R}(\underline{a})\boxtimes\mathcal{R}(\underline{b}).
\end{equation}
\end{itemize}
\end{prop}

As both, the moment- and cumulant series take their values in the same set of non-commutative power series, one should consider them as {\bf two different co-ordinate systems}. In order to relate them, one has to consider a special distribution first.

\begin{df}
The {\bf Zeta distribution} is the linear functional corresponding to $\underline{1}_{\mathcal{A}}=(1_{\mathcal{A}},\dots,1_{\mathcal{A}})$, cf.~(\ref{s-distribution}). It is the unit with respect to component-wise multiplication. 

The {\bf Zeta function} $\operatorname{Zeta}(z_1,\dots, z_s)$, is the moment series of $\underline{1}_{\mathcal{A}}$, given by the power series with all its coefficients equal to $1_R$, i.e.
$$
\operatorname{Zeta}_s:=\operatorname{Zeta}(z_1,\dots,z_s):=\mathcal{M}_{(1_{\mathcal{A}},\dots,1_{\mathcal{A}})}(z_1,\dots, z_s)=\sum_{|w|\geq1}1_R z_w~.
$$
\end{df}
As in~\cite{NS}, we have 
\begin{prop}[Moment-cumulant formula]
\label{Mom_Cumul}
The Zeta function induces a bijection between the free cumulant and the free moment series by right translation, i.e. by $\boxtimes$-multiplication from the right. For any $\underline{a}\in\mathcal{A}^s$, we have 
\begin{equation}
\label{Mom-cum-zeta}
\mathcal{R}(\underline{a})\boxtimes\operatorname{Zeta}_s=\mathcal{M}(\underline{a}),
\end{equation}
i.e. 
\begin{equation*}
\label{RSmu}
\begin{xy}
  \xymatrix{
 &  \mathcal{A}^s\ar[dl]_{\text{$\mathcal{R}$-transform}\quad} \ar[dr]^{\quad\text{$\mathcal{M}$-transform}}   &\\
                  R_+\langle\langle z_1,\dots, z_s\rangle\rangle\ar[rr]^{\boxtimes\operatorname{Zeta}}           &    & R_+\langle\langle z_1,\dots, z_s\rangle\rangle
  }
\end{xy}
\end{equation*}
commutes.
\end{prop}
\begin{df}
\label{Moebius}
The {\bf Möbius function}, $\operatorname{Moeb}_s:=\operatorname{Moeb}(z_1,\dots,z_s)$, is the inverse to $\operatorname{Zeta}_s$ with respect to $\boxtimes$, i.e.
$$
\operatorname{Moeb}_s\boxtimes\operatorname{Zeta}_s=\operatorname{Zeta}_s\boxtimes\operatorname{Moeb}_s=z_1+\dots+z_s. 
$$
\end{df}
The above discussion shows that also in the case of moments we have algebraic morphisms, as there exist the two binary operations  $\boxplus_V$ and $\boxtimes_V$ on the set $R_+\langle\langle z_1,\dots, z_s\rangle\rangle$.
\begin{prop}
\label{V_multi_oper}
Let $(\mathcal{A},\phi)$ be a non-commutative $R-k$-probability space. If $\underline{a}=(a_1,\dots,a_s), \underline{b}=(b_1,\dots,b_s)\in\mathcal{A}^s$ are such that the sets $\{a_1,\dots,a_s\}$ and $\{b_1,\dots,b_s\}$ are combinatorially free, then the
$\mathcal{M}$-transform defines algebraic morphisms
$$
\mathcal{M}:\mathcal{A}^s\rightarrow R_+\langle\langle z_1,\dots, z_s\rangle\rangle,
$$ 
i.e. we have
\begin{itemize}
\item in the {\bf additive} case 
\begin{equation}
\mathcal{M}(\underline{a}+\underline{b})=\mathcal{M}(\underline{a})\boxplus_V\mathcal{M}(\underline{b}),
\end{equation}
\item in the {\bf multiplicative} case
\begin{equation}
\label{R_trafo_mult}
\mathcal{M}(\underline{a}\star\underline{b})=\mathcal{M}(\underline{a})\boxtimes_V\mathcal{M}(\underline{b}).
\end{equation}
\end{itemize}
The relation between $+$ and $\boxplus_V$, and between $\boxtimes$ and $\boxtimes_V$, respectively, is given by 
$$
\boxplus_V=\left(\bullet_1\boxtimes\operatorname{Moeb}_s+\bullet_2\boxtimes\operatorname{Moeb}_s\right)\boxtimes\operatorname{Zeta}_s,
$$
and
$$
\boxtimes_V=\bullet_1\boxtimes\operatorname{Moeb}_s\boxtimes\bullet_2,
$$
with $\bullet_i$, $i=1,2$, denoting the first and the second argument (in moment co-ordinates), respectively.
\end{prop}
\begin{proof}
This follows  from the moment-cumulant formula, by using the identity~(\ref{Mom-cum-zeta}), Definition~\ref{Moebius} and Proposition~\ref{R-trafo-group-morph}.
\end{proof}

\begin{rem}
It is important to point out that the above is not a tautology. Namely, if we had directly used the definition of freeness, we would have eventually found these statements, as Voiculescu did in the one-variable case~\cite{V}. 
However, by using the combinatorial approach via free cumulants, one arrives naturally at these results.
\end{rem}
\subsection{The operation of boxed convolution}
\label{boxconv_sec}
First we introduce two functionals on power series in non-commuting variables with commutative coefficients. Let $R\in\mathbf{cAlg}_k$, $s\in\N$ and $w\in[s]^*$. The {\bf projection} $X_w$ onto the $w$th component is defined by
\begin{equation}
\label{co-ordinate_map}
X_w : R\langle\langle z_1,\dots, z_s\rangle\rangle \rightarrow  R~, \qquad
f \mapsto  X_w(f):=a_w,
\end{equation}
which $R$-linearly assigns to a power series $f$ its $w$th coefficient $a_w$. Equally, for a word $w=(i_1\dots i_n)$, we may write $X_{(i_1\dots i_n)}$. 

Next, we define a restriction operator. Let  $\pi=(V_1,\dots, V_r)\in\operatorname{NC}(n)$ be a non-crossing partition which is composed of $r\geq 1$ blocks, with $V_j=(v_1,\dots, v_m)\subseteq[n]$, and consider a word $w=(i_1\dots i_n)$ such that $1\leq v_1<\dots <v_m\leq n$. 

The {\bf restriction of $w$ onto $V_j$}, is given by
$$
w|V_j:=(i_1\dots i_n)|V_j:=(i_{v_1}\dots i_{v_m})\in\{1,\dots, s\}^m.
$$
\begin{exmpl}
\end{exmpl}
Let $n=7$ and $V=\{1347\}$, which corresponds to $m=4$. Then we have
$$
(i_1\dots i_7)|V=(i_1\bullet_2i_3 i_4\bullet_5\bullet_6 i_7)=(i_1i_3 i_4 i_7),
$$
where $\bullet_j$ indicates that the $j$th position has been deleted.

\begin{df}
\label{def:box-proj-op}
Let $f\in R_+\langle\langle z_1,\dots,z_s\rangle\rangle$, $w$ be a word and $\pi\in\operatorname{NC}(n)$, both as above. There exists an operator $X_{w|\pi}:R_+\langle\langle z_1,\dots,z_s\rangle\rangle\rightarrow R$, given by
\begin{equation}
\label{box-proj-op}
X_{w,\pi}(f):=\prod_{V_j\in\pi}X_{w|V_j}(f)~,
\end{equation}
with $X_{w|V_j}$ as in~(\ref{co-ordinate_map}), and the product taken in $R$ with respect to $\cdot_R$. The {\bf complementary} operator $X_{w,K(\pi)}$ is obtained by replacing the non-crossing partition $\pi$ by its Kreweras complement $K(\pi)$.
\end{df}
The {\bf degree} of the operator $X_{w,\pi}$, corresponding to the monomial $X_{w|V_1}\cdot...\cdot X_{w|V_r}$, is defined as the number of blocks in the partition $\pi$, i.e.
\begin{equation}
\label{deg_oper}
\deg(X_{w,\pi}):=\deg(X_{w|V_1}\cdot...\cdot X_{w|V_r})=r=|\pi|.
\end{equation}
\begin{exmpl}
\end{exmpl}
For $w=(i_1,\dots,i_7)$ and $\pi=\{\{1,7\},\{2,3,5\},\{4\},\{6\}\}$, we have according to~(\ref{box-proj-op}):
$$
X_{(i_1,\dots,i_7),\pi}(f)=X_{(i_1i_7)}(f)\cdot X_{(i_2i_3 i_5)}(f)\cdot X_{(i_4)}(f)\cdot X_{(i_6)}(f),
$$
and so, $\deg(X_{(i_1,\dots,i_7),\pi})=4$.

\begin{lem}
\label{lem:nonlinear}
The operator $X_{w,\pi}$ is not $R$-linear.
\end{lem}
\begin{proof} Let $f,g\in R_+\langle\langle z_1,\dots,z_s\rangle\rangle$ be arbitrary. By~(\ref{box-proj-op}), we then have: 
\begin{eqnarray}
\label{non-lin}
X_{w,\pi}(f+g) & = & \prod_{V_j\in\pi}X_{w|V_j}(f+g)=\prod_{V_j\in\pi}\left(X_{w|V_j}(f)+X_{w|V_j}(g)\right) \\\nonumber
 & \neq & \prod_{V_j\in\pi}X_{w|V_j}(f)+\prod_{V_j\in\pi}X_{w|V_j}(g)=X_{w,\pi}(f)+X_{w,\pi}(g).
\end{eqnarray}
where $``\neq"$ in the last line, holds in general.
\end{proof}

Now, we are ready to introduce the main object of this study. Originally it was derived from the expressions governing the multiplication of free random variables, as it appeared in  Proposition~\ref{mult_free_conv}.

\begin{df}
\label{boxconv_df}
Let $R\in\mathbf{cAlg}_k$. The {\bf boxed convolution} $\boxtimes$, is a binary operation 
\begin{eqnarray}
\label{boxconv}
R_+\langle\langle x_1,\dots, x_s\rangle\rangle\times R_+\langle\langle x_1,\dots, x_s\rangle\rangle&\rightarrow& R_+\langle\langle x_1,\dots, x_s\rangle\rangle,\\\nonumber
(f,g)&\mapsto&f\boxtimes g,
\end{eqnarray}
which for every word $w=(i_1,\dots, i_n)$, and integer $n\geq1$, satisfies:
\begin{eqnarray}
\label{boxed_convolution}
X_{w}(f\boxtimes g)&=&\sum_{\pi\in\operatorname{NC}(|w|)} X_{w,\pi}(f)\cdot_R X_{w,K(\pi)}(g)\\\nonumber
&=&
\sum_{\pi\in\operatorname{NC}(|w|)} X_{(i_1,\dots, i_{n}),\pi}(f)\cdot_R X_{(i_1,\dots, i_{n}),K(\pi)}(g).
\end{eqnarray}
\end{df}
\begin{rem}
The operation~(\ref{boxed_convolution}) is given {\em explicitly} which yields {\em universal polynomials} with {\em integer coefficients}, completely determined by the underlying combinatorics. 

The calculation of a coefficient in~(\ref{boxed_convolution}) of degree $|w|$ depends strictly on the coefficients of terms of degree~$\leq|w|$.
\end{rem}

Let us note that it follows from Proposition~\ref{V_multi_oper} that the same is true for the operations $\boxplus_V$ and $\boxtimes_V$, defined for moment series. In the one-dimensional case, we recover the polynomials originally described by Voiculescu, cf.~\cite{VDN}.

\begin{prop}
\label{prop_non-distributive}
The operation $\boxtimes$ has the following properties.
\begin{enumerate} 
\item The boxed convolution is associative.
\item The operation does {\bf not distribute} in general over point-wise addition,  i.e. for  $f,\tilde{f},g\in R_+\langle\langle x_1,\dots, x_s\rangle\rangle$, we have 
\begin{equation}
\label{not-distributive}
(f+\tilde{f})\boxtimes g\neq f\boxtimes g+\tilde{f}\boxtimes g~.
\end{equation}
\item For $s=1$, it is always commutative, and for $s\geq 2$, non-commutative, in general.
\end{enumerate}
\end{prop}
\begin{proof}
1.   Associativity can be shown in two different ways. Either, by transferring the associativity from the ring structure on $\mathcal{A}^s$ via freeness and the properties of the free cumulants or alternatively, in a purely combinatorial way; for both approaches, cf.~\cite{NS}, p. 275.

2. In order to have distributivity, the definition in~(\ref{boxed_convolution}) requires  $X_{w,\pi}$ to be $R$-linear. However, by Lemma~\ref{lem:nonlinear} this is not the case. 

3. We can write~(\ref{boxed_convolution}) in the form  
\begin{eqnarray*}
X_{w}(f\boxtimes g) & = & \left(X_{w,1_{|w|}}\cdot X_{w,0_{|w|}}+X_{w,0_{|w|}}\cdot X_{w,1_{|w|}}\right) \\
&& +\sum_{\text{pairs $i$}}\left(X_{w,\pi_i}(f)\cdot_R X_{w,K(\pi_i)}(g)+X_{w,K(\pi_i)}(f)\cdot_R X_{w,K^2(\pi_i)}(g)\right) \\
 & & +X_{w,\pi}(f)\cdot_R X_{w,K(\pi)}(g),\qquad\text{additionally, if $|\operatorname{NC}(|w|)|$ is odd},
\end{eqnarray*}
where the sum on the right-hand is over pairs of summands if the number of non-crossing partitions is even, and if it is odd, which happens for $|w|=2^{\ell}-1$, $\ell\in\N^{\times}$, then there is one additional unpaired summand, as shown in the last line of the above expression.

The first pair is always symmetric and if we have an odd number of non-crossing partitions, then the expression can not be symmetric, as the Kreweras complementation map has no fixed point. 

Therefore let us consider the even case.
So, $\boxtimes$ is symmetric if $X_{w,\pi}=X_{w,K^2(\pi)}$ holds for all words $w$ and partitions $\pi$. By Remark~\ref{pi-K^2(pi)} we know that $\pi$ and $K^2(\pi)$ have the same block structure but differ by a cyclic permutation. Therefore, only words which are invariant under such a permutation give symmetric contributions. 

For words over an alphabet consisting of a single letter, i.e. for $s=1$, this is always true and therefore commutativity follows.

For the general case, let us consider an alphabet consisting of at least two letters, e.g., $\{a,b\}$, and the following partition for $n\geq3$:
$$
\pi_0:=((1),(2,\dots,n)),
$$
and so
$$
K(\pi_0)=((1,n),(2),\dots,(n-1))\qquad\text{and}\quad K^2(\pi_0)=((1,\dots,n-1),(n)).
$$
For a word $\underbrace{ab\dots b}_{\text{$n$-letters}}$, $n$ even, and $\underbrace{ab\dots a}_{\text{$n$-letters}}$, $n$ odd, we get
\begin{equation}
\label{comb_table}
\begin{tabular}{c|c|c}$n$ & $\pi_0$ & $K^2(\pi_0)$ \\\hline even & $a|b\dots b$ & $a\dots a|b$ \\\hline odd & $a|b\dots a$ & $a\dots b|a$\end{tabular}
\end{equation}
in which case $X_{w,\pi_0}\neq X_{w,K^2(\pi_0)}$, i.e. for $s\geq2$ the boxed convolution is not commutative in general. 
\end{proof}

If we use  infinite vectors, indexed by words, instead of non-commutative power series, the $\boxtimes$-multiplication of two such vectors $(a_1,\dots, a_s,\dots)$ and $(b_1,\dots, b_s,\dots)$,
takes for the first $s$ entries, according to~(\ref{boxed_convolution}), the form
\begin{equation}
\label{mult_box_string}
(a_1,\dots, a_s,\dots)\boxtimes (b_1,\dots, b_s,\dots)=(a_1\cdot b_1,\dots, a_s\cdot b_s,\dots),
\end{equation}
further with the unit given by the sequence  
$$
(\underbrace{1_R,\dots,1_R}_{\text{$s$-times}}):=(\underbrace{1_R,\dots,1_R}_{\text{$s$-times}},0,0,0,\dots),
$$
and the $\boxtimes$-inverse of a vector $(r_1,\dots, r_s,\dots)$ with $r_i\in R^{\times}$, for $i\in[s]$, by an expression of the form    
\begin{equation}
\label{vect_inv}
(r^{-1}_1,\dots, r^{-1}_s,\dots).
\end{equation}

\section{The affine group schemes}
\subsection{The groups associated to the boxed convolution}
Here we introduce several functors from $\mathbf{cAlg}_k$ to $\mathbf{Set}$ of a special form. For an integer $s\geq1$, let
\begin{eqnarray*}
\mathfrak{M}^s,\mathfrak{G}^s,\dots:\mathbf{cAlg}_k &\rightarrow & \mathbf{Set} \\
R & \mapsto & R^{[s]^*},R^{[s]^*_+},\dots
\end{eqnarray*}
be maps from commutative $k$-algebras into the set of (infinite) strings over the algebra, indexed by (arbitrary) words over the alphabet $[s]$, or equivalently, non-commutative power series in $s$ non-commuting indeterminates. 

The natural morphisms are defined as follows: for $\varphi\in\Hom_{\mathbf{cAlg}_k}(R,R')$, we have
$\mathfrak{M}(\varphi):(\alpha_w)\mapsto(\varphi(\alpha_w))$, with $a_w\in R$, and equivalently  for $\mathfrak{G}^s$, etc.

We let
\begin{eqnarray*}
\mathfrak{M}^s(R)& := & \big\{\sum_{|w|\geq 1}\alpha_w z_w~|~\alpha_w\in R, w\in [s]^*\big\}\cong R^{[s]_+^*},\\
\mathfrak{G}^s(R) & := & \{(r_1,\dots, r_s,\dots, r_w,\dots)~|~r_i\in R^{\times}, i\in[s], r_w\in R, |w|\geq 2\},\\
\mathfrak{G}^s_+(R) & := & \{(1_1,\dots, 1_s,\dots, r_w,\dots)~|~r_w\in R, |w|\geq 2\},
\end{eqnarray*}
and the finite-dimensional restrictions ({\bf truncations}) of $\mathfrak{M}^s(R)$, $\mathfrak{G}^s(R)$ and $\mathfrak{G}^s_+(R)$, respectively. For $n\geq1$, we have
\begin{eqnarray*}
(\mathfrak{M}^s)_n(R)&:=& \big\{\sum_{1\leq|w|\leq n }\alpha_w z_w\big\},\\
(\mathfrak{G}^s)_n(R) & := & \big\{\sum_{1\leq|w|\leq n }\alpha_w z_w~| \alpha_i\in R^{\times}\quad\text{for $i=1,\dots, s$}\big\},\\
(\mathfrak{G}^s_+)_n(R) & := & \big\{\sum_{1\leq|w|\leq n }\alpha_w z_w~|~\alpha_i=1_R\quad\text{for $i=1,\dots, s$}\big\}.
\end{eqnarray*} 
For an integer $n\geq1$, let $\pr_n:\mathfrak{G}^s\rightarrow(\mathfrak{G}^s)_n$ denote the {\bf natural truncations}. Conversely, let $\iota_n:(\mathfrak{G}^s)_n\hookrightarrow\mathfrak{G}^s$ denote the {\bf natural embeddings} such that
$$
\sum_{1\leq|w|\leq n}\alpha_wz_w\mapsto  \sum_{1\leq|w|}\alpha_wz_w,\quad\text{with $\alpha_w=0$\quad for $|w|>n$}.
$$

For every integer $n\geq 1$, we have the strict inclusions of sets
$$
(\mathfrak{G}^s_+)_n(R)\subsetneq (\mathfrak{G}^s)_n(R)\subsetneq (\mathfrak{M}^s)_n(R),
$$
and 
$$
\mathfrak{G}^s_+(R)\subsetneq \mathfrak{G}^s(R)\subsetneq \mathfrak{M}^s(R),
$$
respectively. 
\begin{prop}
\label{Group_prop}
For $s\in\N^{\times}$, $R\in\mathbf{cAlg}_k$ and the binary operation $\boxtimes$, the following holds:
\begin{enumerate}
\item $(\mathfrak{M}^s(R),+)$ is, with component-wise addition,  an abelian group, and additive neutral element $0_{\mathfrak{M}^s(R)}$, the series with all coefficients being equal to zero.
\item $(\mathfrak{M}^s(R),\boxtimes)$ is an associative multiplicative monoid, with unit the polynomial
$$
1_{\mathfrak{M}^s(R)}=1_Rz_1+\cdots+1_R z_s.
$$
\item $(\mathfrak{M}^s(R),+,\boxtimes)$, is an associative and {\bf not distributive} unital ring. For $s\geq2$, it is not commutative.
\item $\mathfrak{G}^s(R)$ is a group with neutral element
$
1_{\mathfrak{G}^s(R)}=1_Rz_1+\cdots+1_R z_s.
$\\
For $s=1$, it is abelian and for $s\geq2$, non-abelian.
\end{enumerate}
\end{prop}
\begin{rem}
In order to simplify notation, we shall also use $f_w:=X_w(f)$, $f_{w,\pi}:=X_{w,\pi}(f)$ and $f_{w,K(\pi)}:=X_{w,K(\pi)}(f)$.
\end{rem}
\begin{proof}

1. This follows from the general properties of arbitrary power series with coefficients in a commutative ring.   

2. Associativity was shown in Proposition~\ref{prop_non-distributive}, Point 1. 
First we show that $1_{\mathfrak{M}^s(R)}$ is a unit. 
For $|w|\geq1$, it follows form~(\ref{boxed_convolution}), that  
$(1_{\mathfrak{M}^s(R)}\boxtimes f)_w=(1_{\mathfrak{M}^s(R)})_{w,0_{|w|}}\cdot f_{w,1_{|w|}}=f_w$, with $(1_{\mathfrak{M}^s(R)})_{w,0_{|w|}}=1_R$, and $0_{|w|}$ being the Kreweras complement of $1_{|w|}$; the other contributions are zero. Similarly, for $(f\boxtimes 1_{\mathfrak{M}^s(R)})_w=f_{w,1_{|w|}}\cdot (1_{\mathfrak{M}^s(R)})_{w,0_{|w|}}=f_w$, and all other contributions being again zero. Uniqueness is then proved as usual. 

3. This follows directly from Proposition~\ref{prop_non-distributive} and the second statement in the present proposition. 

4. The form of the unit is clear. The property of being abelian or non-abelian, respectively, was already shown in Proposition~\ref{prop_non-distributive}, point~$3.$ It remains to show the existence of an inverse, which will be done in Proposition~\ref{co-inv-boxed}.
\end{proof}

For every $s\geq1$, the family
$((\mathfrak{G}^s)_{i}, \pr_{ij})$, indexed by the directed set $(\N,\leq)$, is an inverse system of functors, with the natural restriction homomorphisms 
$$
\pr_{ij}: (\mathfrak{G}^s)_{j}\rightarrow(\mathfrak{G}^s)_{i},\quad i\leq j, 
$$ 
satisfying  
\begin{itemize}
\item $\pr_{ii}=\id_{\mathfrak{G}^s_i}$,  $i\in\N$, and 
\item $\pr_{ik}=\pr_{ij}\circ\pr_{jk}$, $i\leq j\leq k$.
\end{itemize}

\begin{df}
\label{def:n-torus}
An algebraic group scheme is called an {$n$-\bf torus} if it is isomorphic to $\mathbb{G}^n_m\cong \underbrace{\mathbb{G}_m\times\cdots\times \mathbb{G}_m}_{n-\text{times}}$ for some
$n\in\N$. 
\end{df} 
So, for $R\in\mathbf{cAlg}_k$, we have $\mathbb{G}^n_m(R)\cong\prod_{i=1}^n R^{\times}$, with component-wise multiplication. 

\begin{prop}
\label{basic_properties}
For all $s\in\N^{\times}$, $R\in\mathbf{cAlg}_k$, and the binary operation $\boxtimes$, the following holds:
\begin{enumerate}
\item $\mathfrak{G}^s(R)$ is a group, which is the {\bf semi-direct product} of the {\bf normal subgroup} $\mathfrak{G}^s_+(R)$ and the {\bf $s$-torus}
$(\mathfrak{G}^s)_1(R)\cong\mathbb{G}_m^s(R)$, i.e.
$$
\mathfrak{G}^s(R)=\mathfrak{G}^s_1(R)\ltimes (\mathfrak{G}^s_+)(R).
$$
Equivalently, the exact sequence
\[
\begin{xy}
  \xymatrix{
                             0_R\ar[r]&\mathbb{G}_m^s(R)\ar[r]^{\iota} &\mathfrak{G}^s(R)\ar[r]^p &\mathfrak{G}^s_+(R)\ar[r] &1_R~.
               }
\end{xy}
\]
holds.
\item $\mathfrak{G^s}(R)$ and $\mathfrak{G}^s_+(R)$ are groups, filtered in ascending order by the subgroups $(\mathfrak{G}^s)_n(R)$ and $(\mathfrak{G}^s_+)_n(R)$, respectively, i.e. 
$$
\forall n: (\mathfrak{G}^s)_n(R)\subsetneq  (\mathfrak{G})_{n+1}(R) \qquad\text{and}\quad \bigcup_{n\in\N^{\times}} (\mathfrak{G}^s)_n(R)=\mathfrak{G}^s(R),
$$
and similarly for $(\mathfrak{G}^s_+)_n(R)$.
\item $\mathfrak{G}^s(R)$ and $\mathfrak{G}^s_+(R)$ are {\bf pro-groups} (projective limits of finite-dimensional groups), i.e. 
$$
\mathfrak{G}^s(R)=\varprojlim_{n\in\N} (\mathfrak{G}^s)_n(R)\qquad\text{and}\quad\mathfrak{G}^s_+(R)=\varprojlim_{n\in\N} (\mathfrak{G}^s_+)_n(R)~.
$$
\end{enumerate}
\end{prop}
\begin{proof}
The second and the third statement are clear. It remains to show the first one.
Let $g\in\mathfrak{G}(R)$ and $h\in\mathfrak{G}_+(R)$. As vectors they are of the form $g=(r_1,\dots, r_s,\dots)$, $r_i\in R^{\times}$, $i=1,\dots, s$, and $h=(1,\dots,1,h_w,\dots)$ with $h_w$ arbitrary for $|w|\geq2$. Then for $g\boxtimes h\boxtimes g^{-1}$, according to~(\ref{mult_box_string}), (\ref{vect_inv}) and associativity, we have
\begin{eqnarray*}
(r_1,\dots, r_s,\dots)\boxtimes(1,\dots,1,h_w,\dots)\boxtimes(r_1^{-1},\dots, r_s^{-1},\dots)&=&(r_1\cdot r_1^{-1},\dots, r_s\cdot r_s^{-1},\dots)\\
&=&(1,\dots,1,\tilde{h}_w,\dots)
\end{eqnarray*}
with $\tilde{h}_w\in R$ for $|w|\geq2$, which shows normality. 

As $\mathfrak{G}_1^s(R)=\{(r_1,\dots, r_s):\text{$r_i\in R^{\times}$ for $i\in[s]$} \}$, the  boxed convolution becomes component-wise multiplication, and so the isomorphism with an $s$-dimensional torus, is clear from Definition~\ref{def:n-torus}. 

It follows from the definitions that $\mathfrak{G}^s_1(R)\cap\mathfrak{G}^s_+(R)=\{1_{\mathfrak{G}^s(R)}\}$. For an arbitrary $f\in\mathfrak{G}^s(R)$, consider $\pr_1(f)\in\mathfrak{G}^s_1(R)$ and define $\iota_1((\pr_1(f))^{-1})\boxtimes f\in\mathfrak{G}^s_+(R)$. We obtain the factorisation
$$
f=\iota_1(\pr_1(f))\boxtimes (\iota_1((\pr_1(f))^{-1})\boxtimes f),
$$
i.e. $\mathfrak{G}^s=\mathfrak{G}^s_1\boxtimes\mathfrak{G}^s_+$, which proves the claim.
\end{proof}

In the case of moments we have 
\begin{prop} 
\label{Prop:mom-groups}
For every $s\in\N^{\times}$ and $R\in\mathbf{cAlg}_k$, the following holds:
\begin{enumerate}
\item $(\mathfrak{M}^s(R),\boxplus_V)$ is an {\em abelian} group with neutral element $\underline{0}=(0,0,0,0,\dots)$. 
\item $(\mathfrak{G}^s(R),\boxtimes_V)$ is a group with neutral element $\underline{1}=(1,1,1,1,\dots)=\operatorname{Zeta}_s$. For $s=1$, it is {\em abelian} and for $s\geq2$, {\em non-abelian}. 
\item The $\boxtimes$-multiplication from the right with $\operatorname{Zeta}_s$ induces  a natural isomorphism of groups, such that
$$
(\mathfrak{M}^s(R),+)\cong(\mathfrak{M}^s(R),\boxplus_V)\quad\text{and}\quad (\mathfrak{G}^s(R),\boxtimes)\cong(\mathfrak{G}^s(R),\boxtimes_V).
$$
\end{enumerate}
\end{prop}
\begin{proof}
This is a consequence of Propositions~\ref{Mom_Cumul} and \ref{V_multi_oper}.
\end{proof}
\begin{df}
Let $G$ be an algebraic group with identity $e_G$. A finite sequence of subgroups $(G_i)_{i=0}^n$, which starts with $G$ and ends with $e_G$, i.e.
$$
G=G_0\supset G_1\supset\dots\supset G_n=\{e_G\},
$$
is called a {\bf subnormal series} if each subgroup is normal in the preceding subgroup, i.e. $G_{i+1}\lhd G_i$, for all $i$.
It is called a {\bf normal series} if every $G_i$ is a normal subgroup of $G$, i.e. $G_{i}\lhd G$, for all $i$.
\end{df}
\begin{df}
\label{def:solvable}
An algebraic group $G$ is called {\bf solvable} if it admits a subnormal series 
$$
G\supset G_1\supset\dots\supset G_n=\{e_G\},
$$
such that for all $i=1,\dots, n-1$ the quotients $G_{i}/G_{i+1}$ are abelian groups. It is {\bf split solvable} if the factor groups are either $\mathbb{G}_a$ or $\mathbb{G}_m$.
\end{df}
\begin{rem} 
Split solvable algebraic groups are smooth and connected, cf. e.g.~\cite{W}.
\end{rem}
For fixed $s,n,j\in\N$, let $(\mathfrak{G}^s_+)_{n,j}(R)$ be the subset of power series with coefficients $a_w$, such that $a_w=0$, for $2\leq|w|\leq j\leq n$, and more generally, 
$(\mathfrak{G}^s_+)_{\infty,j}(R)$ those series for which $a_w=0$, for $2\leq|w|\leq j$.
\begin{prop}
\label{split_normal_group}
For $k$ a field, let $R\in\mathbf{cAlg}_k$ and $j=2,\dots, n$. The $(\mathfrak{G}^s_+)_{n,j}(R)$ are normal subgroups of $(\mathfrak{G}^s_+)_{n}(R)$. Set $G_0(R):=(\mathfrak{G}^s_+)_n(R)$ and $G_i(R):=(\mathfrak{G}^s_+)_{n,i+1}(R)$, for $i=1,\dots, n-1$. The $G_i(R)$ form a {\bf split} subnormal series, i.e. subsequent quotients satisfy 
\begin{equation}
\label{}
G_i(R)/G_{i+1}(R)\cong\mathbb{G}^{s^{i+1}}_a(R)\quad\text{for $i=0,\dots,n-1$}.
\end{equation}
\end{prop}
\begin{proof}
Let us show normality first. 
For any $f\in\mathfrak{G}^s(R)$, we have
$$
(f\boxtimes f^{-1})_w=
\begin{cases}
      & 1_k\quad \text{for $|w|=1$ }, \\
      & 0_k\quad \text{for $|w|\geq 2$}.
\end{cases}
$$
Let $h\in(\mathfrak{G}^s_+)_{n,j}(R)$ and $f\in(\mathfrak{G}^s_+)_{n}(R)$. Then by~(\ref{boxed_convolution}), for
$2\leq|w|\leq j$,  we have 
$$
(f\boxtimes h)_w=f_w\cdot 1+1\cdot\underbrace{h_w}_{=0}+\sum_{\substack{\pi\in\operatorname{NC}(|w|)\\  \pi\neq0_{|w|},1_{|w|}}} f_{w,\pi}\cdot \underbrace{h_{w,K(\pi)}}_{=0}=f_w,
$$
as the coefficients in the last sum always involve words $w'$ of length $2\leq|w'|\leq j-1$, and hence they are zero. 
Thererfore $((f\boxtimes h)\boxtimes f^{-1})_w=0$ for $1<|w|\leq j$ and $1$ for $|w|=1$. This proves normality.

Next we shall determine the cosets.  Let $f,g\in(\mathfrak{G}^s_+)_{n,j}(R)$. For $|w|=j+1$ we have $(f\boxtimes g)_w=f_w+g_w=(g\boxtimes f)_w$, i.e. the corresponding coefficients are simply added.

Let $\tilde{f}\in(\mathfrak{G}^s_+)_{n,j}(R)$ be of the form: $\tilde{f}_w$ is arbitrary for $|w|=j+1$ and $0$ for $|w|\geq j+2$. The dimension of the space of these power series is $s^{j+1}$.

Now we look at the equivalence class of such an $\tilde{f}$. Chose $f\in(\mathfrak{G}^s)_{n,j}$ such that $f_w=\tilde{f}_w$ for $|w|=j+1$ but otherwise arbitrary for all $|w|\geq j+2$. We show that we can uniquely solve the equation 
$$
f_w=(h\boxtimes \tilde{f})_w\quad\text{for $|w|\geq j+2$}, 
$$
i.e. we can find an $h\in(\mathfrak{G}^s_+)_{n,j+1}(R)$ which satisfies the above identities. For $|w|=j+1$ we have $h_w=0$, and for $|w|=j+2$ we have 
$$
f_w=(h\boxtimes\tilde{f})_w=h_w+\underbrace{\tilde{f}_w}_{=0}+\sum_{\substack{\pi\in\operatorname{NC}(|w|)\\  \pi\neq0_{|w|},1_{|w|}}} \underbrace{h_{w,\pi}\cdot\tilde{f}_{w,K(\pi)}}_{=0}=h_w,
$$
as the summands in the right hand sum are zero as $h_{w,\pi}$ is, and the other zero contribution comes from the assumption on $\tilde{f}$. Therefore, $h_w=f_w$ for $|w|=j+2$. Next we proceed inductively, by looking at words of length $|w|+1$. By abuse of notion we write $w+1$.
$$
f_{w+1}=(h\boxtimes\tilde{f})_{w+1}=h_{w+1}+\underbrace{\tilde{f}_{w+1}}_{=0}+\sum_{\substack{\pi\in\operatorname{NC}(|w|+1)\\  \pi\neq0_{|w|+1},1_{|w|+1}}} h_{w+1,\pi}\cdot\tilde{f}_{w+1,K(\pi)},
$$
but, by induction the expressions involved in $h_{w+1,\pi}\cdot\tilde{f}_{w+1,K(\pi)}$ are already determined, and so
$$
h_{w+1}=f_{w+1}-\sum_{\substack{\pi\in\operatorname{NC}(|w|+1)\\  \pi\neq0_{|w|+1},1_{|w|+1}}} h_{w+1,\pi}\cdot\tilde{f}_{w+1,K(\pi)}.
$$
Therefore, the factor group $G_i(R)/G_{i+1}(R)$, $i=0,\dots,n-1$, is isomorphic to the $s^{i+1}$-dimensional additive group $\mathbb{G}_a(R)$.
\end{proof}
\begin{cor}
\label{solvable}
For $k$ a field, $R\in\mathbf{cAlg}_k$ and $s,n\in\N$, the groups $\mathfrak{G}^s_n(R)$ are split solvable.
\end{cor}
\begin{proof}
Let $G(R):=(\mathfrak{G}^s)_n(R)$, and consider the sequence $G_1(R):=(\mathfrak{G}^s_+)_n(R)$ and $G_i(R):=(\mathfrak{G}^s_+)_{n,i+1}(R)$, $i\geq 1$. From Propositions~\ref{basic_properties} and~\ref{split_normal_group} we know that $G_i(R)$ defines a subnormal series with the quotients $G_i(R)/G_{i+1}(R)$ being additive groups for all $i\geq1$. 

It remains to look at $G(R)/G_1(R)$. By Proposition~\ref{basic_properties} we have $G(R)/G_1(R)=\mathbb{G}^s_m(R)$, as claimed.
\end{proof}
Alternatively, we can investigate the properties of the groups under consideration as follows. For $x,y$ being elements of a group $G$, let $[x,y]:=xyx^{-1}y^{-1}$ be the corresponding {\bf commutator}, and for two subgroups $A,B\subseteq G$, denote by $[A,B]$ the subgroup of $G$ generated by all $[a,b]$, $a\in A, b\in B$.
\begin{df}
Let $G$ be a group. The {\bf derived series} $(D^i(G))_{i\in\N}$ is the inductively defined sequence of iterated commutators:
$$
D^0(G):=G,\quad D^{i+1}(G):=[D^i(G),D^i(G)],\quad i\in\N.
$$
\end{df}
Equivalently, a group $G$ is solvable, cf.~Definition~\ref{def:solvable}, if the derived series terminates in the neutral element $e_G$.
It is known, cf. e.g.~\cite{Hum,W} that if $G$ is an algebraic group, then $D^1(G)=[G,G]$ is a {\em closed normal} subgroup of $G$, and {\em connected} if $G$ is. Further,  these properties hold true, as can be shown by induction, for all $D^i(G)$. 

Now, a direct calculation gives
\begin{equation}
D^1(\mathfrak{G}^s)=[\mathfrak{G}^s,\mathfrak{G}^s]=\mathfrak{G}^s_+,
\end{equation}
which yields another proof of the first statement in Proposition~\ref{basic_properties}.
\begin{exmpl}
\end{exmpl}
Consider the alphabet $\{a,b\}$, i.e. the situation for $s=2$. For $f,g\in \mathfrak{G}^s_+$, we have the infinite strings
$$
\left.\begin{array}{lcrrrrrrc}f & = & 1, & 1, & f_{aa}, & f_{ab}, & f_{ba}, & f_{bb}, & \cdots \\f^{-1} & = & 1, & 1, & -f_{aa}, & -f_{ab}, & -f_{ba}, & -f_{bb}, & \cdots \\
g & = & 1, & 1, & g_{aa}, & g_{ab}, & g_{ba}, & g_{bb}, & \cdots\\
g^{-1} & = & 1, & 1, & -g_{aa}, & -g_{ab}, & -g_{ba}, & -g_{bb}, & \cdots \end{array}\right.
$$
For the commutator $[f,g]=f\boxtimes g\boxtimes f^{-1}\boxtimes g^{-1}$, we obtain by noticing that $[f,g]_{ij}=f_{ij}+g_{ij}-f_{ij}-g_{ij}=0$, for $i,j\in\{a,b\}$:
$$
\left.\begin{array}{lcrrrrrrrc}[f,g] & = & 1, & 1, & 0, & 0, & 0, &0, & [f,g]_{aaa},\dots
\end{array}\right.
$$
i.e.
$$
[\mathfrak{G}^2_+,\mathfrak{G}^2_+]=(\mathfrak{G}^2_+)_{\infty,2},
$$
which also holds for $s\geq3$.
\subsection{The co-ordinate algebra}
All the previous functors are, as we have seen, group valued. More precisely, they define pro-algebraic groups. In this section we assume, for simplicity, $k$ to be a field.
\begin{prop}
\label{coordinate_rings}
For $n,s\in\N^{\times}$ let $N=N(n,s):=\frac{s^{n+1}-1}{s-1}=\varhash[s]^n_+$, be the number of non-empty words of length $\leq n$, over the alphabet $[s]$. We have
\begin{enumerate}
\item
$\mathfrak{M}^s_n(k)$ corresponds to the affine algebraic variety $\mathbb{A}^N(k)$, with co-ordinate ring $k[X_1,\dots, X_N]$.
\item $\mathfrak{G}^s_n(k)$ is isomorphic to a quasi affine variety, given as the cartesian product of principal open sets 
$$
D(X_1)\times\dots\times D(X_s)\times \mathbb{A}^{N-s}.
$$ 
It is isomorphic to the affine algebraic variety
$$
k^2/\langle X_1 Y-1\rangle\times\cdots\times k^2/\langle X_s Y-1\rangle\times k^{N-s},
$$ 
with co-ordinate ring isomorphic to
$$
k[X_1^{\pm1},\dots X_s^{\pm1}, X_{s+1},\dots, X_N].
$$
\item $(\mathfrak{G}_+^s)_n(k)$ is isomorphic to an affine variety in $\mathbb{A}^N(k)$, given as the zero set of the ideal $\langle X_1-1,\dots, X_s-1\rangle).$
Its co-ordinate ring is
$$
\frac{k[X_1,\dots,X_s,\dots, X_N]}{\langle X_1-1,\dots, X_s-1\rangle}\cong k[\bar{X}_w~|~2\leq|w|\leq n].
$$ 
\end{enumerate}
\end{prop}
\begin{proof}
We use the natural embeddings
$$
\mathfrak{G}^s_n(k)\hookrightarrow\mathbb{A}^N(k),
$$
etc., which are well defined up to permutations of the co-ordinates. 
\begin{enumerate}
\item The first claim is clear.
\item 
The principal open set $D(X)\subset \mathbb{A}^1$ is isomorphic to the algebraic variety $k[X,Y]/\langle XY-1\rangle$, from which the first part of the claim follows. For $k[X,Y]/\langle XY-1\rangle$, consider $k(X)$, the field of fractions of $k[X]$ and its subalgebra $k[X]_X$, which is the {\bf localisation} at $X$. It is isomorphic to the algebra of {\bf Laurent polynomials} $k[X,X^{-1}]$ in the commuting variables $X$ and $X^{-1}$, which satisfy $X X^{-1}=1$. The surjective $k$-algebra homomorphism $h:k[X,Y]\rightarrow k[X,X^{-1}]$, defined by
$$
X\mapsto X,\quad Y\mapsto X^{-1},
$$ induces the isomorphism 
$$
k[X,Y]/\underbrace{\langle XY-1\rangle}_{=\ker(h)}\rightarrow k[X,X^{-1}].
$$
From this we get, by using repeatedly $k[V_1\times V_2]\cong k[V_1]\otimes k[V_2]$ for $V_1\subset \mathbb{A}^m$ and $V_1\subset \mathbb{A}^n$ being algebraic sets,
$$
k[X_1,Y]/\langle X_1Y-1\rangle\times\cdots\times k[X_s,Y]/\langle X_sY-1\rangle\cong k[X_1,X_1^{-1},\dots, X_s,X_s^{-1}].
$$
\item In $k[X]/\langle X-1\rangle$, let $\bar{X}$ denote the equivalence class of $X$. We have $\bar{X}=\bar{1}$. For $k[X_w|1\leq|w|\leq n]$ we get by dividing by the  prime ideal $\langle X_1-1,\dots, X_s-1\rangle:$ 
$$
X_w\mapsto\begin{cases}
      & \bar{1}\qquad \text{if $|w|=1$ }, \\
      & \bar{X}_w\quad\text{otherwise},
\end{cases}
$$
which proves the last claim.
\end{enumerate}
\end{proof}
Let us first determine the Hopf algebra structure on the co-ordinate ring of $(\mathfrak{M}^s)_n(R)$, as discussed in Section~\ref{Sec:GSch}. It is enough to look at the generators $X_w$ and then extend it to the rest algebraically.

The $k$-algebra structure on $R[X_w:1\leq|w|\leq n]$ is given as follows: for $f,g\in(\mathfrak{M}^s)_n(R)$, we have:
\begin{itemize}
\item the unit:
$\eta:R\rightarrow R[X_w:1\leq|w|\leq n]$ with $R$ being mapped onto the polynomials of degree zero, and with $1_R$ as unit element,
\item $R$-linear combinations $(a X_w+bX_{w'})(f):= a X_w(f)+b X_{w'}(f)$, for $a,b\in R$, 
\item the commutative product $(X_w\cdot X_{w'})(f):=X_w(f)\cdot X_{w'}(f)$.
\end{itemize}
For the co-algebra structure we have: 
\begin{itemize}
\item co-product: $\Delta(X_w)=X_w\otimes 1+1\otimes X_w$, as $\Delta X_w(f,g):=X_w(f+g)$, 
\item co-unit: $\varepsilon(X_w)=0_R$, as $\varepsilon(X_w):=X_w(0_{(\mathfrak{M}^s)_n(R)})$,
\item co-inverse (antipode): $S(X_w)=-X_w$, as $S(X_w)(f):=X_w(-f)$.
\end{itemize}
\begin{rem}
The above co-product is {\em co-commutative} and furthermore all the elements $X_w$ are {\em primitive}.
\end{rem}

\begin{prop}
For all $s,n\in\N^{\times}$, $((\mathfrak{M}^s)_n,+)$ is naturally isomorphic to the additive group $\mathbb{G}^{N}_a$ in $N=N(n,s)=\varhash[s]^n_+$-variables. By taking inverse limits, we obtain: 
\[
\begin{xy}
  \xymatrix{
    ((\mathfrak{M}^s)_n,+) \ar[d]^{\varprojlim}\ar[r]^{\quad\sim}& \mathbb{G}^{N}_a\ar[d]^{\varprojlim}\\ (\mathfrak{M}^s,+)\ar[r]^{\quad\sim} & \mathbb{G}^{\infty}_a
               }
\end{xy}
\]

\end{prop}
Let us determine now the Hopf algebra structure on   
$$
k[X_1^{\pm1},\dots X_s^{\pm1}, X_{s+1},\dots, X_N]
$$ 
which by Proposition~\ref{coordinate_rings} is the co-ordinate algebra of $(\mathfrak{G}^s)_n(R)$. It is enough to look at the generators $X_w$, $1\leq |w|\leq n$, and to take additionally the relations for Laurent polynomials into account. In particular we have
$$
X_i^{-1}(f)=X_i(f)^{-1}\quad\text{for $i=1,\dots,s$}.
$$
So, the $k$-algebra structure is essentially the same as above, however, the co-algebra structure is different.

We have:
\begin{itemize}
\item for the co-product: $\Delta X_w(f,g):=X_w(f\boxtimes g)$, which according to~(\ref{boxed_convolution}) translates into
\begin{equation}
\label{co-product}
\Delta X_w=\sum_{\pi\in\operatorname{NC}(|w|)} X_{w,\pi}\otimes X_{w,K(\pi)}~,
\end{equation}
where $X_{w,\pi}$ and $X_{w,K(\pi)}$ are as in~Definition~\ref{def:box-proj-op}.

\item For the co-unit: $\varepsilon(X_w):=X_w(1_{(\mathfrak{G}^s)_n})$ we get, according to Proposition~\ref{Group_prop},
$$
\varepsilon(X_w)=\begin{cases}
      & 1\quad \text{for $w=(i), i\in\{1,\dots,s\}$}, \\
      &  0\quad \text{for $|w|\geq2$}.
\end{cases}
$$
\end{itemize}
This defines the structure of a {\bf bi-algebra} on $R[X_w| 1\leq|w|\leq n]$.  
\begin{rem}
\label{rem:group-like}
The elements $X_i$, $i\in\{1,\dots s\}$, are {\bf group-like}, i.e. they satisfy 
\begin{equation}
\label{grouplike}
\Delta(X_i)=X_i\otimes X_i\qquad\text{and}\qquad\varepsilon(X_i)=1~.
\end{equation}
\end{rem}

\begin{exmpl}
\end{exmpl}
Let us list the co-product for the co-ordinate functions $X_w$ corresponding to the words $w=(i_1),(i_1i_2), (i_1i_2i_3)$ and $(i_1i_2i_3i_4)$ with $i_j\in[s]$ for $j=1,2,3,4$. We have for 

$|w|=1$:
\begin{eqnarray*}
\Delta(X_{i_1}) & = & X_{i_1}\otimes X_{i_1}
\end{eqnarray*}
$|w|=2$:
\begin{eqnarray*}
\Delta(X_{i_1i_2}) & = & X_{i_1i_2}\otimes X_{i_1}X_{i_2}+X_{i_1}X_{i_2}\otimes X_{i_1i_2}
\end{eqnarray*}
$|w|=3$:
\begin{eqnarray*}
\Delta(X_{i_1i_2i_3}) & = & X_{i_1i_2i_3}\otimes X_{i_1}X_{i_2}X_{i_3}+X_{i_1}X_{i_2i_3}\otimes X_{i_1i_3}X_{i_2}\\
&& +X_{i_1i_2}X_{i_3}\otimes X_{i_1}X_{i_2i_3}+X_{i_1i_3}X_{i_2}\otimes X_{i_1i_2}X_{i_3}\\
&& +X_{i_1}X_{i_2}X_{i_3}\otimes X_{i_1i_2i_3}
\end{eqnarray*}
$|w|=4$:
\begin{eqnarray*}
\Delta(X_{i_1i_2i_3i_4}) &=& X_{i_1i_2i_3i_4}\otimes X_{i_1}X_{i_2}X_{i_3}X_{i_4}+
 X_{i_1}X_{i_2}X_{i_3}X_{i_4}\otimes X_{i_1i_2i_3i_4}\\
&& +X_{i_1}X_{i_2i_3i_4}\otimes X_{i_1i_4} X_{i_2}X_{i_3}+X_{i_1i_3i_4}X_{i_2}\otimes X_{i_1i_2} X_{i_3}X_{i_4}\\
&&+X_{i_1i_2i_4}X_{i_3}\otimes X_{i_1}X_{i_2i_3} X_{i_4}+X_{i_1i_2i_3}X_{i_4}\otimes X_{i_1}X_{i_2}X_{i_3i_4} \\
&&+X_{i_1i_2}X_{i_3i_4}\otimes X_{i_1} X_{i_3}X_{i_2i_4}+
X_{i_1i_4}X_{i_2i_3}\otimes X_{i_1i_3} X_{i_2}X_{i_4}\\
&&+X_{i_1}X_{i_2}X_{i_3i_4}\otimes X_{i_1i_2i_4} X_{i_3}+
X_{i_1}X_{i_2i_3}X_{i_4}\otimes X_{i_1i_3i_4} X_{i_2}\\
&&+X_{i_1i_2}X_{i_3}X_{i_4}\otimes X_{i_1} X_{i_2i_3i_4}+
X_{i_1}X_{i_2i_4}X_{i_3}\otimes X_{i_1i_4} X_{i_2i_3}\\
&&+X_{i_1i_4}X_{i_2}X_{i_3}\otimes X_{i_1i_2i_3} X_{i_4}+
X_{i_1i_3}X_{i_2}X_{i_3}\otimes X_{i_1i_2} X_{i_3i_4}
\end{eqnarray*}

\begin{exmpl}
\end{exmpl}
Let us check the relation $(\id\otimes \varepsilon)\Delta=\id=(\varepsilon\otimes\id)\Delta$, explicitly. For $|w|\geq2$ we have:
\begin{eqnarray*}
\varepsilon\otimes{\id}(\Delta(X_w))&=&\varepsilon\otimes{\id}\left(X_w\otimes X_{i_1}\cdots X_{i_{|w|}}+X_{i_1}\cdots X_{i_{|w|}}\otimes X_w+\sum_{\pi\ne 0,1} X_{w,\pi}\otimes X_{w,K(\pi)}\right)\\
& =& \underbrace{\varepsilon(X_w)}_{=0}\otimes(X_{i_1}\cdots X_{i_{|w|}})+\varepsilon(X_{i_1}\cdots X_{i_{|w|}})\otimes X_w+\sum_{\pi\ne 0,1}\underbrace{\varepsilon(X_{w,\pi})}_{=0}\otimes X_{w,K(\pi)}\\
&=&\underbrace{\varepsilon(X_{i_1})}_{=1}\cdots \underbrace{\varepsilon(X_{i_{|w|}})}_{=1}\otimes X_w\cong X_w=\id(X_w),
\end{eqnarray*}
and similarly for ${\id}\otimes\varepsilon(\Delta(X_w))$.

In order to have a Hopf algebra and not only a bi-algebra, we need a co-inverse, i.e. we require that 
\begin{equation}
\label{}
S(X_w)(f)=X_w(f^{-1}).
\end{equation} 
By using Definition~\ref{boxconv_df}, we can calculate it recursively. 
From Proposition~\ref{coordinate_rings} and the relation $S(X_i)(f)=X_i(f^{-1})=(X_i(f))^{-1}$, we get
\begin{equation}
\label{antipode_inverse}
S(X_i)=X^{-1}_i,
\end{equation}
which satisfies $X^{-1}_i X_i=1_R=X_i X_i^{-1}$.

For $|w|=2$, we obtain by a direct calculation, using~(\ref{antipode_inverse}) and the definition of the boxed convolution
$$
{S(X_{i_1i_2})=-X_{i_1}^{-1}X_{i_2}^{-1}\cdot X_{i_1i_2}\cdot X_{i_1}^{-1}X_{i_2}^{-1}=-X_{i_1}^{-1}X_{i_2}^{-1}\cdot X_{i_1i_2}}\cdot S(X_{i_1}X_{i_2}).
$$

\begin{prop}[Antipode of the boxed convolution]
\label{co-inv-boxed}
Let $w=(i_1\dots i_n)\in[s]^*$ with $n\geq 2$. The antipode corresponding to the boxed convolution $\boxtimes$, is given by the {\em recursive} relation
\begin{equation}
\label{antipode}
S(X_w)=S(X_{(i_1\dots i_n)})=-(X^{-1}_{i_1}\cdots X^{-1}_{i_n})\cdot\sum_{\substack{\pi\in\operatorname{NC}(|w|)\\  \pi\neq0_{|w|}}}X_{w,\pi}\cdot S(X_{w,K(\pi)}).
\end{equation}
\end{prop}
In the case of the groups $(\mathfrak{G}_+^s)_n$ we have to take instead of  $X_i$, $i=1,\dots,s$, its equivalence class $\bar{X}_i$, which according to Proposition~\ref{coordinate_rings} amounts to dividing the previous identities by the ideal $\langle X_1-1,\dots X_s-1\rangle$.

Relation~(\ref{box-proj-op}) becomes then $\bar{X}_{w,\pi}=\prod_{V_j\in\pi} \bar{X}_{w|V_j}$, and 
the definition of the co-product~(\ref{co-product}) remains the same modulo the above replacement, i.e. $X_i\mapsto \bar{X}_i=\bar{1}$. 
\begin{exmpl}
\end{exmpl}
For $|w|=2$, we have
\begin{equation*}
\Delta(\bar{X}_{i_1i_2})=\overline{\Delta(X_{i_1i_2})} =  \bar{X}_{i_1i_2}\otimes 1+1\otimes \bar{X}_{i_1i_2},
\end{equation*}
and for $|w|=3$ 
\begin{eqnarray*}
\Delta(\bar{X}_{i_1i_2i_3})=\overline{\Delta(X_{i_1i_2i_3})} & = & \bar{X}_{i_1i_2i_3}\otimes 1+1\otimes \bar{X}_{i_1i_2i_3}\\
&& +\bar{X}_{i_2i_3}\otimes \bar{X}_{i_1i_3}
 +\bar{X}_{i_1i_2}\otimes \bar{X}_{i_2i_3}+\bar{X}_{i_1i_3}\otimes \bar{X}_{i_1i_2},
\end{eqnarray*}
where $\overline{\Delta(X_w)}$ denotes the equivalence class of $\Delta(X_w)$.
\begin{rem}
For $|w|=2$, the elements $\bar{X}_w$ are {\bf primitive}, i.e. they satisfy $\Delta=\bar{X}_w\otimes 1+1\otimes\bar{X}_w$ and $\varepsilon(\bar{X}_w)=0$.
\end{rem}
\begin{df}
\label{df:coconnected}
A Hopf algebra $A$ is {\bf co-connected}  if there exists an ascending filtration  $A_0\subseteq A_1\subseteq A_2\subseteq\dots$ of $A$ by subspaces $A_j$, $j\in\N$, such that 
$$
A_0=k, \qquad\bigcup_{j=0}^{\infty} A_j=A\qquad\text{and}\qquad\Delta(A_n)\subseteq\sum_{j=0}^n A_j\otimes A_{n-j},\quad\text{for all $n\geq0$}.
$$ 
\end{df}
Let us assign to $X_w$ and $\bar{X}_w$, respectively, the homogeneous weight $|w|-1$, for $|w|\geq1$, and to the constant function $1$, weight zero. As next we shall define the following families of subspaces: 
for $A:=k[X^{\pm1}_1,\dots,X^{\pm1}_s, X_w:|w|\geq2]$, and $j\in\N$, let
$$
C_j(k):=\operatorname{span}_k\{{X}_{w_1}\cdots{X}_{w_n}~|~|w_1|+\dots+|w_n|-n\leq j\}\subset A,  
$$
be the $k$-linear span of monomials of total degree at most $j$, and similarly for $\tilde{A}:=k[X_w:|w|\geq2]$
$$
\tilde{C}_j(k):=\operatorname{span}_k\{\bar{X}_{w_1}\cdots\bar{X}_{w_n}~|~|w_1|+\dots+|w_n|-n\leq j\}\subset \tilde{A}.  
$$
By Remark~\ref{rem:group-like} we then have  
$$
C_0(k)=\operatorname{span}_k\{{X}^{\pm 1}_{i_1}\cdots{X}^{\pm 1}_{i_n}~|~i_l\in[s], l\in\N\}=\operatorname{span}_k\{\text{group-like elements in $A$}\},
$$
and $\tilde{C}_0(k)=k$. In both cases we obtain a strictly ascending filtration by the non-negative integer numbers, and with the corresponding unions being equal to the algebras $A$ and $\tilde{A}$, respectively. 
\begin{prop}
\label{prop:grad_conn} The following holds:
\begin{itemize}
\item the Hopf algebra defined on $k[X^{\pm1}_1,\dots, X^{\pm1}_s,X_w:|w|\geq2]$ is $\N$-filtered, with $C_0(k)$ being equal to the $k$-linear span of the group-like elements. 
\item The Hopf algebra defined on $k[\bar{X}_w:|w|\geq 2]$
is co-connected.
\end{itemize}
\end{prop}
\begin{proof}
The following properties remain to be shown:

For ${X}_{w_1}\cdots{X}_{w_p}\in C_m(k)$ and ${X}_{\tilde{w}_1}\cdots{X}_{\tilde{w}_q}\in C_n(k)$, we have ${X}_{w_1}\cdots{X}_{w_p}\cdot \bar{X}_{\tilde{w}_1}\cdots\bar{X}_{\tilde{w}_q}\in C_{m+n}(k)$ and therefore $C_m(k)\cdot C_n(k)\subset C_{m+n}(k)$, and the same holds true for $\tilde{A}$.

Let us show now the third property in Definition~\ref{df:coconnected}.
For $n:=|w|-1$, the weighted degree of $\bar{X}_{w,\pi}$ of a partition $\pi\in\operatorname{NC}(|w|)$ is
$$
\operatorname{wdeg}(\bar{X}_{w,\pi})=|w|-|\pi|=:j,
$$
and for the Kreweras complement, using relation~(\ref{part-Krew_rel}): $|\pi| + |K(\pi)|=|w|+1$,
$$
\operatorname{wdeg}(\bar{X}_{w,K(\pi)})=|w|-|K(\pi)|=|\pi|-1=n-j.
$$
So, the total weighted degree $\operatorname{wdeg}(\bar{X}_{w,\pi})+\operatorname{wdeg}(\bar{X}_{w,K(\pi)})$ is equal to $|w|-1=n$, and the same property is satisfied in $A$.
\end{proof}

By taking projective limits we obtain
\begin{thm}
\label{first_main_thm}
Let $s\in\N^{\times}$ and $R\in\mathbf{cAlg}_k$. The following functors define affine group schemes, such that:
\begin{itemize}
\item $\mathfrak{G}^s$ is represented by
\begin{equation}
\label{}
k[X^{\pm 1}_1,\dots, X^{\pm 1}_s, X_w:|w|\geq 2];
\end{equation}
which is filtered but not connected.
\item $(\mathfrak{G}^s)_1$ is represented by  
$$ k[X_1,X^{-1}_1,\dots,X_s,X^{-1}_s],$$
and it consists of {\em group-like} elements.
\item $\mathfrak{G}^s_+$ is represented by the co-connected Hopf algebra
\begin{equation}
\label{}
k[\bar{X}_w:|w|\geq 2].
\end{equation}
For $s=1$, it is co-commutative but not for $s\geq2$. 
\end{itemize}
\end{thm}

As mentioned in Section~\ref{Sec:GSch}, there is a correspondence between subgroups and Hopf subalgebras. 
Therefore, the subgroup in $\mathfrak{G}^s$ corresponding to the augmentation ideal $\ker(\varepsilon)=k[X_w:|w|\geq2]$ is exactly the normal subgroup $\mathfrak{G}^s_+$, 
and the {characters} of $\mathfrak{G}^s$ are exactly the group-like elements, i.e. $k[X_1^{\pm1},\dots, X_s^{\pm1}]$.

\subsection{Formal group laws and Lie algebra associated to freeness}
We have $k[\bar{X}_w:|w|\geq2]\subset k[[\bar{X}_w:|w|\geq2]]$. 
By abuse of notation, we shall also write $\bar{Y}_w:=1\otimes \bar{X}_w$ and $\bar{X}_w=\bar{X}_w\otimes 1$. 
\begin{prop}[Associated formal group law]
Let $s\in\N^{\times}$, $w\in[s]^*_+$ and $R\in\mathbf{cAlg}_k$. The infinite-dimensional formal group law 
\begin{equation}
\label{free_g_law}
F_{\mathfrak{G}_+^s}:=(F_w(\underline{X},\underline{Y}))_{|w|\geq2},
\end{equation}
corresponding to $(\mathfrak{G}_+^s(R),\boxtimes)$,  is given by 
\begin{eqnarray*}
F_w(\underline{X},\underline{Y})&=&\bar{X}_w\otimes 1+1\otimes \bar{X}_w+\sum_{\substack{\pi\in\operatorname{NC}(|w|)\\  \pi\neq0_{|w|},1_{|w|}}} \bar{X}_{w,\pi}\otimes \bar{X}_{w,K(\pi)}\\
&=& \bar{X}_w+\bar{Y}_w\quad \mod\deg(2).
\end{eqnarray*}
For $s=1$, it is commutative and for $s\geq2$, non-commutative.
\end{prop}
\begin{proof}
The form of the group law follows directly from the general correspondence between smooth formal groups and the group law its co-product induces, cf. Section~\ref{Sec:GSch}.

The statement about commutativity follows from Lemma~\ref{Lie_algebra_boxed} and Theorem~\ref{Q-theorems}.
\end{proof}
Let us now determine the Lie algebra corresponding to the formal group law~(\ref{free_g_law}).

In~(\ref{deg_oper}) we assigned to $X_{w,\pi}$ a monomial of degree $|\pi|$. Here, we have to modify it, as we work with the equivalence classes $\bar{X}_w$. For a partition $\pi=(V_1,\dots V_r)$, we let
\begin{equation}
\label{monom_order}
\deg(\bar{X}_{w,\pi}):=\begin{cases}
      & r\quad \text{if $|V_i|\geq 2$, for all $i=1,\dots,r$},\\
      & r-\hash\{\text{blocks $V_i$ with $|V_i|=1$}\},\,\,\text{otherwise}.
\end{cases}
\end{equation}

\begin{lem}
\label{Lie_algebra_boxed}
For $s\geq2$ and $|w|\geq3$, the bilinear form $B_w(\underline{X},\underline{Y})$ associated to $F_w(\underline{X},\underline{Y})$ in~(\ref{free_g_law}), is in general not symmetric, and  
$
B_{\mathfrak{G}^s_+}=(B_w)_{|w|\geq2}
$
is never symmetric.  For $s=1$, $B_w$ is symmetric for all $|w|\geq1$, and therefore also  $B_{\mathfrak{G}^1_+}$. 
\end{lem}
\begin{proof}
The bilinear form $B_w(\underline{X},\underline{Y})$ would be symmetric if
\begin{equation}
\label{bilinear_form}
\bar{X}_{w,\pi}\otimes \bar{X}_{w,K(\pi)}=\bar{X}_{w,K(\pi)}\otimes \bar{X}_{w,K^2(\pi)}
\end{equation}
were to hold. For $s\geq2$ and $|w|\geq3$, this is not so by~Proposition~\ref{prop_non-distributive}, but it is true for $s=1$ and $|w|\geq1$.

Next, let us show that the vector of bilinear forms $B_{\mathfrak{G}^s_+}$ is neither zero nor symmetric.

An expression $\bar{X}_{w,\pi}\otimes \bar{X}_{w,K(\pi)}$ gives a monomial of degree two if $\bar{X}_{w,\pi}$ and $\bar{X}_{w,K(\pi)}$ are of degree one. 

As in the proof of Proposition~\ref{prop_non-distributive}, consider for $|w|\geq3$, the partition $\pi_0:=((1),(2,\dots,|w|))$. Then $K(\pi_0)=((1,|w|),(2),\dots(|w|-1))$ and $K^2(\pi_0)=((1,2,\dots,|w|-1),|w|)$. Hence 
$$
\deg(\bar{X}_{w,\pi_0}\otimes\bar{X}_{w,K(\pi_0)})=2\qquad\text{and}\qquad\deg(\bar{X}_{w,K(\pi_0)}\otimes\bar{X}_{w,K^2(\pi_0)})=2.
$$
Therefore, 
$$
B_w(\underline{X},\underline{Y})=\bar{X}_{w,\pi_0}\otimes \bar{X}_{w,K(\pi_0)}+\bar{X}_{w,K(\pi_0)}\otimes \bar{X}_{w,K^2(\pi_0)}+\text{other quadratic terms},
$$ is not zero. In general, it is also not symmetric, as we show next.

As in the proof of Proposition~\ref{prop_non-distributive}, consider a word of the form $w_1:=ab\dots b$, if $|w|$ is odd,  and $w_{2}:=ab\dots a$ if $w$ is even, and then $\bar{X}_{w_{i},\pi_0}\neq \bar{X}_{w_{i},K^2(\pi_0)}$, $i=1,2$. Therefore, the corresponding bilinear form $B_{w_i}$ is not symmetric, and hence neither is $B_{\mathfrak{G}^s_+}$.

Finally, it follows from the above that $B_{\mathfrak{G}^1_+}$ is non-zero and symmetric. 
\end{proof}
\begin{exmpl}
\label{Bilinear_form}
\end{exmpl}
Let us give the bilinear form for $|w|\leq 4$, explicitly. For

$|w|=2$: 
$$
B_{i_1i_2}(\underline{X},\underline{Y})=0~.
$$

$|w|=3$:
\begin{eqnarray*}
B_{i_1i_2i_3}(\underline{X},\underline{Y})& = &\bar{X}_{i_2i_3}\otimes \bar{Y}_{i_1i_3} +\bar{X}_{i_1i_2}\otimes \bar{Y}_{i_2i_3}+\bar{X}_{i_1i_3}\otimes \bar{Y}_{i_1i_2}
\end{eqnarray*}
$|w|=4$:
\begin{eqnarray*}
B_{i_1i_2i_3i_4}(\underline{X},\underline{Y})
&=& \bar{X}_{i_2i_3i_4}\otimes \bar{Y}_{i_1i_4}+
\bar{X}_{i_1i_3i_4}\otimes \bar{Y}_{i_1i_2}+X_{i_1i_2i_4}\otimes \bar{Y}_{i_2i_3}  +\bar{X}_{i_1i_2i_3}\otimes \bar{Y}_{i_3i_4}\\ 
&&+\bar{X}_{i_3i_4}\otimes \bar{Y}_{i_1i_2i_4}+
\bar{X}_{i_2i_3}\otimes \bar{Y}_{i_1i_3i_4} +
\bar{X}_{i_1i_2}\otimes  \bar{Y}_{i_2i_3i_4}+
\bar{X}_{i_1i_4}\otimes \bar{Y}_{i_1i_2i_3}
\end{eqnarray*}

\begin{thm} 
\label{main-iso}
For a $\Q$-algebra $R$, the additive group $\mathbb{G}_a^{\N}(R)$ and the group $(\mathfrak{G}^1_+(R),\boxtimes)$ are {\em isomorphic}. For $s\geq 2$, $(\mathfrak{G}^s_+,\boxtimes)$ and $\mathbb{G}^{\N}_a$ are not isomorphic (as the first group is not commutative). 
\end{thm}
\begin{proof}
This follows from Lemma~\ref{Lie_algebra_boxed} and Theorem~\ref{Q-theorems}.
\end{proof}

\begin{rem}
\begin{itemize}
\item This answers a major question in Free Probability, which had been open since the late 1980s, cf.~\cite{S97,N,NS}. The present proof is more conceptual than the one we gave in ~\cite{FMcK2}.
\item The concrete form of the isomorphism will be given in Theorem~\ref{Thm:main_one_dim}.
\end{itemize} 
\end{rem}

\subsection{Faithful representations and the general $\mathcal{S}$-transform}

Here we generalise the $\mathcal{S}$-transform to arbitrary dimensions. In this section we assume $k$ to be a field. 
\begin{prop} 
For $s,n\in\N$, the groups $((\mathfrak{G}^s)_n,\boxtimes)$ are faithfully representable as closed subgroups of some $\operatorname{GL}_N$, $N=N(n,s)$. To every $\mathfrak{G}^s$ corresponds a faithful family of finite-dimensional representations.
\end{prop}
\begin{proof}
By Theorem~\ref{first_main_thm}, the above are affine algebraic group schemes. As they can be linearised, cf.~\cite{W}, Theorem 3.4, the first part of the claim follows. 

The second part can be shown as follows. By~Proposition~\ref{basic_properties} we have $\mathfrak{G}^s=\varprojlim_{n\in\N} (\mathfrak{G}^s)_n$, i.e. it is an inverse limit of algebraic groups (i.e. of finitely generated $k$-algebras). So, for every $n\in\N$, we can choose a faithful finite-dimensional representation of $(\mathfrak{G}^s)_n$, which in turn can be regarded as a not-necessarily faithful representation of $\mathfrak{G}^s$ itself. However, the whole family consists of faithful items.  
\end{proof}

Let us now introduce the types of matrices we shall need. 

The set of {\bf diagonal matrices}  
$$
\mathbb{T}_n=\{\operatorname{diag}(t_1,\dots,t_n)~|~t_i\in k^{\times}\},
$$ 
in $\operatorname{GL}_n(k)$ is called the {\bf torus}, and the set of invertible {\bf upper-triangular} matrices 
$$
\mathbb{B}_n(k):=\{(a_{ij})_{i,j}\in\operatorname{GL}_n(k)~ | \text{$a_{ij}=0$ for $i>j$}\},
$$ the {\bf Borel subgroups}. 

Let 
$$
\mathbb{B}_{\infty}(k):=\varprojlim_{n\in\N} \mathbb{B}_n(k)\quad\text{and}\quad\mathbb{U}_{\infty}(k):=\varprojlim_{n\in\N} \mathbb{U}_n(k),
$$
denote the respective projective limits, which are infinite-dimensional upper-triangular matrices. Nevertheless, matrix multiplication is  well-defined, as for every entry only finitely many operations are involved when calculated. 
\begin{df}
Let $A$ be a square matrix $A$ and $E$ the unit matrix. $A$ is called {\bf unipotent} if $A-E$ is {\bf nilpotent},  i.e. there exists an integer $n\geq1$ such that $(A-E)^n=0$.\\
An element $r$ in a ring $R$ with unit $1_R$, is called unipotent, if $(r-1_R)$ is nilpotent, i.e. $\exists n\geq 1$, $(r-1_R)^n=0_R$.
\end{df}
\begin{df} 
\label{def:unipotent}
An affine group scheme $G$ is called {\bf unipotent} if every non-zero linear representation  of $G$ has a non-trivial {\em fixed vector}, i.e. there exists a non-zero $v\in V$ such that $\rho(v)=v\otimes 1$, in the associated $k[G]$-comodule.
\end{df}

More generally we have
\begin{df}
\label{df:trigonalisable}
An affine group scheme $G$  is {\bf triangulable / trigonalisable} if every non-zero linear representation of $G$ has a {\em one-dimensional  invariant subspace}, i.e. in the associated comodule $(V,\rho)$ there exists a non-zero $v\in V$ such that $\rho(v)=v\otimes a$, for some $a\in k[G]$.
\end{df}

The {upper-triangular} matrices with constant diagonal $1$ are unipotent
$$
\mathbb{U}_n(k):=\left(\begin{array}{ccccc}1 & * & \dots  & * \\0 & 1 & \ddots & \vdots  \\ \vdots& \ddots & \ddots & *  \\ 0& \cdots &  0& 1 \end{array}\right)
$$
with $*$ denoting arbitrary entries from $k$. 
In particular, the group $\mathbb{G}^n_a(k)$ is unipotent as it is represented by 
$$
\left(\begin{array}{cccccc}1 & x_1 & \cdots  & \cdots  &x_n\\0 & 1 & 0 & \dots  &0\\ \vdots& \ddots & \ddots & \ddots &\vdots\\ 0& \cdots &  0& 1 &0\\
0&0&\cdots&0&1\end{array}\right)
$$
with $x_i\in k$, $i=1,\dots, n$. 

\begin{prop} 
\label{prop:being_unipo}
Let $s,n\in\N$ and $k$ be a field. For all $n\geq1$,
\begin{itemize}
\item the algebraic groups $(\mathfrak{G}^s_+)_n(k)$ are unipotent and $\mathfrak{G}^s_+(k)$ is pro-unipotent,
\item the algebraic groups $(\mathfrak{G}^s)_n(k)$ and its inverse limit $\mathfrak{G}^s(k)$, are trigonalisable.
\end{itemize}
\end{prop}
\begin{proof}
It follows from Proposition~\ref{prop:grad_conn} that for every $n\in\N^{\times}$, the co-ordinate algebra $k[(\mathfrak{G}^s_+)_n(k)]$ is co-connected. Therefore, by~\cite{W}, Theorem 8.3, the first claim follows. 

Equivalently, by Proposition~\ref{split_normal_group} every $(\mathfrak{G}^s)_n(k)$ admits a split subnormal series with the respective quotients being isomorphic to $\mathbb{G}_a(k)$. Then by~\cite{Mil}, Corollary~2.13, the statement follows.

The second claim follows from the second statement in Proposition~\ref{prop:grad_conn}, which is equivalent to the existence of a closed embedding of $(\mathfrak{G}^s)_n(k)$ as a subgroup of some $\mathbb{B}_N(k)$, $N=N(s,n)$, cf.~\cite{W}, p. 72. 

Alternatively, if we assume $k$ also to be algebraically closed, the statement can be deduced from the {\bf Lie-Kolchin Triangularisation Theorem}, cf.~\cite{W}, Theorem~10.2, as by Corollary~\ref{solvable}, the affine algebraic groups $(\mathfrak{G}^s)_n(k)$ are solvable and by~\cite{W}, Theorem~3.4 linearisable. 

The above proofs assumed {\em algebraic} groups. However, the statements about inverse limits of unipotent and trigonalisable groups being unipotent and trigonalisable, resepctively, follow from Definitions~\ref{def:unipotent} and~\ref{df:trigonalisable}, respectively. 
\end{proof}
\begin{thm}
\label{Higher-S-trafo}
Let $k$  be a field, $s,n\in\N^{\times}$ and $N=N(n,s)$ a positive integer. There exist {\em faithful} finite-dimensional representations $\rho_n$, for
\begin{itemize}
\item $((\mathfrak{G}^s)_n(k),\boxtimes)$ as closed subgroups of the {\bf Borel groups} $\mathbb{B}_N(k)$,
\item $((\mathfrak{G}^s_+)_n(k),\boxtimes)$ as closed subgroups of the {\bf unipotent groups} $\mathbb{U}_N(k)$, 
\item $((\mathfrak{G}^s)_1(k),\boxtimes)$ as a group of {\bf multiplicative type} $\mathbb{T}_s(k)$.
\end{itemize}
\end{thm}
Note that $\rho_n$ is a natural transformation, so that we should indicate the dependence on the field $k$, e.g. as $\rho_n(k)$, which we omitted here. 
\begin{proof}
All of the above groups are by Proposition~\ref{coordinate_rings} algebraic.  Therefore, both, 
the first and second statement follow from from Proposition~\ref{prop:being_unipo}, and the third is clear from the definitions. 
\end{proof}
\begin{rem}
Let us note that in characteristic $0$, the Lie algebra of an unipotent algebraic group is nilpotent. 
\end{rem}
\begin{df}[$\mathcal{S}$-transform]
Let $k$ be a field. A minimal faithful representation,up to isomorphism,
$$
\rho:(G^s(k),\boxtimes)\rightarrow\mathbb{B}_{\infty}(k),
$$ 
is called the {\bf $s$-dimensional $\mathcal{S}$-transform}. 

If $(\mathcal{A},\phi)$ is a non-commutative $k$-probability space and $\underline{a}=(a_1,\dots,a_s)\in\mathcal{A}^s$, then the $\mathcal{S}$-transform of $\underline{a}$ is defined as
\begin{equation}
\label{}
\mathcal{S}(\underline{a}):=\rho(\mathcal{R}(\underline{a})).
\end{equation}
\end{df}
We have
\begin{thm}
Let $(\mathcal{A},\phi)$ be a non-commutative $k$-probability space over a field $k$, and  $\underline{a}=(a_1,\dots,a_s), \underline{b}=(b_1,\dots,b_s)\in\mathcal{A}^s$ such that the sets $\{a_1,\dots,a_s\}$ and $\{b_1,\dots,b_s\}$ are combinatorially free. 
The $\mathcal{S}$-transform is a faithful morphism
$$
\mathcal{S}:\mathcal{A}^s\rightarrow \varprojlim_N\mathbb{B}_N(k),
$$
which satisfies
\begin{equation}
\label{S_transform}
\mathcal{S}(\underline{a}\star\underline{b})=\mathcal{S}(\underline{a})\cdot\mathcal{S}(\underline{b}).
\end{equation}
\end{thm}
\begin{proof}
By Proposition~\ref{prop:being_unipo}, the $\mathcal{S}$-transform exists and is well-defined, up to isomorphism. So, the identity~(\ref{S_transform}) follows from the definitions and Proposition~\ref{R-trafo-group-morph}, as:
$$
\mathcal{S}(\underline{a}\star\underline{b})=\rho(\mathcal{R}(\underline{a}\star\underline{b}))=\rho(\mathcal{R}(\underline{a})\boxtimes\mathcal{R}(\underline{b}))=\rho(\mathcal{R}(\underline{a}))\cdot\rho(\mathcal{R}(\underline{b}))=\mathcal{S}(\underline{a})\cdot\mathcal{S}(\underline{b}).
$$ 
\end{proof}
\begin{rem}
The main benefit of having the general $\mathcal{S}$-transform is that we can linearise our problems by studying linear groups, i.e. matrices. However, concrete representations involve arbitrary choices, e.g. for vector bases.  

In particular, all linear representations can be derived from the regular representation. 

Also, it is clear from~Proposition~\ref{Prop:mom-groups} that equivalent statements can be made in terms of the moment series. 
\end{rem}

\section{Generalised one-dimensional free harmonic analysis}
Historically, the one-dimensional case played a central role in free-probability theory, in particular in what is called ``free harmonic analysis". Here we shall re-examine it in the light of our results.

For $s=1$, the following simplifications occur as compared to the higher dimensional case:
$$
\begin{array}{c c c}  R\langle\langle z_1,\dots z_s\rangle\rangle & \rightarrow & R[[z]] \\ \text{non-commutative} & \rightarrow & \text{commutative}\\
w\in[s]^{*} & \rightarrow & |w|\in\N\end{array}
$$
In particular the functors  
$\mathfrak{G}^1,\mathfrak{G}^1_+:\mathbf{cAlg_k}\rightarrow\mathbf{Set}$,  
have corresponding distributions $R\mapsto\Hom_{k,1}(k[z],R)$, such that
\begin{eqnarray*}
\mathfrak{G}^1(R) & \cong & \{\mu:k[z]\rightarrow R~|~\mu(z)\neq 0,\, \mu(1_k)=1_R\}, \\
\mathfrak{G}^1_+(R)&\cong&\{\mu:k[z]\rightarrow R~|~\mu(z)=1,\, \mu(1_k)=1_R \}. 
\end{eqnarray*}

Next, let us determine the changes for the boxed convolution. 
For  $|w|\geq1$, two series $f(z)=\sum_{j=1}^{\infty}\alpha_j z^j$ and $g(z)=\sum_{j=1}^{\infty}\beta_j z^j$, and a partition $\pi=\{V_1,\dots, V_p\}\in\operatorname{NC}(|w|)$, with Kreweras complement $K(\pi)=\{W_1,\dots, W_q\}$, the operators $X_{w,\pi}$ and $X_{w,K(\pi)}$, as in~(\ref{box-proj-op}), become
\begin{eqnarray*}
\label{}
X_{\pi}(f)&=&\alpha_{|V_1|}\cdots\alpha_{|V_p|}~,\\
X_{K(\pi)}(f)&=&\beta_{|W_1|}\cdots\beta_{|W_q|}~,\\
\end{eqnarray*}
and therefore the operation $\boxtimes$, as defined in~(\ref{boxed_convolution}),  specialises to
\begin{equation}
\label{}
X_{|w|}(f\boxtimes g)=\sum_{\pi\in\operatorname{NC}(|w|)}X_{\pi}(f)\otimes X_{K(\pi)}=\sum_{\begin{subarray}{c}
       \pi\in\operatorname{NC}(|w|)\\\pi=\{V_1,\dots V_p\}\\K(\pi)=\{W_1,\dots W_q\}
       \end{subarray}}\alpha_{|V_1|}\cdots\alpha_{|V_p|}\beta_{|W|_1}\cdots\beta_{|W|_q},
\end{equation}
with the product on the right being evaluated in the ring $R$.

Recall, that  ${\OO}_R:=R[[z]]$ is the ring of formal power series with $R$-coefficients, and $\Aut({\OO}_R)$ denotes its group of automorphisms, with the group operation given by composition of power series. Finally, $\aut_+({\OO}_R)$ is the subgroup consisting of automorphisms with first derivative equal to one.

For $R\in\mathbf{cAlg}_k$, the {\bf units} of the ring or formal power series $R[[z]]$ are given by
$$
R[[z]]^{\times}=R^{\times}(1+zR[[z]]),
$$
and as rings we have
$
\Lambda(R)\subsetneq R[[z]]^{\times}.
$
Further, it follows that
$$
(\mathfrak{M}^1(R),+)\cong (zR[[z]],+)\cong\mathbb{G}_a^{\N}(R).
$$
  
\begin{lem}
The following are equal as sets:
\begin{eqnarray*}
\label{}
\mathfrak{G}^1(R) & = & \aut({\OO}_R), \\
\mathfrak{G}^1_+(R) & = & \aut_+({\OO}_R), 
\end{eqnarray*}
but the corresponding groups are not isomorphic, i.e.
\begin{eqnarray*}
\label{}
(\mathfrak{G}^1(R),\boxtimes) & \ncong & (\aut({\OO}_R),\circ), \\
(\mathfrak{G}^1_+(R),\boxtimes) &\ncong& (\aut_+({\OO}_R),\circ).
\end{eqnarray*}
\end{lem}
\begin{proof}
The statements about the sets follow directly from the definitions. Those about the groups follow from the fact that the first group is abelian and the second non-abelian. Therefore they can not be isomorphic.
\end{proof}

To every $f\in\mathfrak{G}^1(R)$ corresponds a formal power series of the form
$$
f(z)=a_1z+a_2z^2+a_3 z^3+\dots\qquad a_1\in R^{\times},
$$
and therefore it has an inverse $f^{-1}(z)$ with respect to composition of power series, which can be obtained from the co-inverse of the Faà di Bruno Hopf algebra, cf.~(\ref{FdB_coniverse}).

\begin{df}
Let $(\mathcal{A}_k,\phi)$ be a non-commutative $R-k$-probability space. For every $R\in\mathbf{cAlg}_k$, the {\bf $\mathcal{F}$-transform} is a morphism
$$
\mathcal{F}_R:\mathfrak{G}^1(R)\rightarrow R[[z]]^{\times},
$$
such that for $f\in\mathfrak{G}^1(R)$ it is given by
\begin{equation}
\label{F-trafo}
\mathcal{F}_R(f):=\frac{1}{z}f^{-1}(z),
\end{equation}
and for $a\in\mathcal{A}_k$, with $\phi(a)\in R^{\times}$, by
\begin{equation}
\label{F-trafo_cumul}
\mathcal{F}_R(a):=\frac{1}{z}\mathcal{R}_a^{-1}(z),
\end{equation}
where $\mathcal{R}_a^{-1}(z)$ is the compositional inverse. 
\end{df}
The following two statements can be deduced from the original theorems by Nica and Speicher, cf.~\cite{NS,S97}.
\begin{prop}
Let us denote by $\bullet[[z]]^{\times}$ the functor $R\mapsto R[[z]]^{\times}$.
The $\mathcal{F}$-transform is a natural isomorphism from $\mathfrak{G}^1$ to $\bullet[[z]]^{\times}$, i.e. every $\mathcal{F}_R$ is an isomorphism of abelian groups, such that 
 \[
\begin{xy}
  \xymatrix{
  \left(\mathfrak{G}^1_+(R),\boxtimes\right)\ar[d]_{\operatorname{incl.}}\ar[rr]_{{\mathcal{F}_R}}^{\sim}&&\left(\Lambda(R),\cdot\right)\ar[d]_{\operatorname{incl.}}\\
\left(\mathfrak{G}^1(R),\boxtimes\right)\ar[rr]_{\mathcal{F}_R}^{\sim}&&\left(R[[z]]^{\times},\cdot\right) 
     }
\end{xy}
\]
commutes, where $\operatorname{incl.}$ is the natural inclusion morphism. 
\end{prop}

\begin{prop} 
Let $(\mathcal{A},\phi)$ be a non-commutative $R-k$-probability space and consider random variables $a,b\in \mathcal{A}$ which are combinatorially free, with $\phi(a),\phi(b)\in R^{\times}$. The $\mathcal{F}_R$-transform of $a\cdot b$ satisfies:
$$
\mathcal{F}_R(a\cdot b)=\mathcal{F}_R(a)\cdot \mathcal{F}_R(b)~.
$$
\end{prop}

We know from Section~\ref{examples} and Theorem~\ref{first_main_thm}, that $\Lambda$ is represented by the Hopf algebra $\mathbf{Symm}$ and $\mathfrak{G}^1_+$ by $k[\bar{X}_{|w|}:|w|\geq2]$, respectively.

As  for every $R\in\mathbf{cAlg}_k$, the map $\mathcal{F}_R$ is an isomorphism of abelian groups,  it induces by Yoneda's Lemma a morphism of Hopf algebras. 

For $k[\bar{X}_*]:=k[\bar{X}_2,\bar{X}_3,\bar{X}_4,\dots]$ and $k[h_*]:=k[h_1,h_2,h_3,\dots]$, we have the following isomorphisms 
\[
\begin{xy}
  \xymatrix{
    \Hom_{\mathbf{cAlg}}(k[\bar{X}_*],k[\bar{X}_*]) \ar[r]& \mathfrak{G}_+^1(k[\bar{X}_*])\ar[d]^{\mathcal{F}_k(k[\bar{X}_*])}\\            \Hom_{\mathbf{cAlg}}(k[h_*],k[\bar{X}_*])\ar[r] & \Lambda(k[\bar{X}_*])
               }
\end{xy}
\]
The universal element is 
$
z+\sum_{j=2}^{\infty} \bar{X}_j z^j, 
$
and therefore we obtain by looking at $\mathcal{F}(k[\bar{X}_*])(z+\sum_{j=2}^{\infty} \bar{X}_j z^j)$, the expression for the Hopf algebra isomorphism.
\begin{prop}
The co-ordinate change between the Hopf algebra $k[\bar{X}_{|w|}:|w|\geq2]$ and $\mathbf{Symm}_k$ is given by 
\begin{equation}
\label{coord-chg_symm}
h_n\mapsto S_{\operatorname{FdB}}(\bar{X}_{n+1}), 
\end{equation}
where $S_{\operatorname{FdB}}$ is the co-inverse of the {\em Faà di Bruno Hopf algebra}~(\ref{FdB_coniverse}).
\end{prop}
For the first four co-ordinates they are given as:
\begin{eqnarray*}
h_1& = & \bar{X}_2, \\
h_2 & = & 2\bar{X}^2_1-\bar{X}_2, \\
h_3 & = & 5\bar{X}^3_1-5\bar{X}_1\bar{X}_2+\bar{X}_3, \\
h_4 & = & 14\bar{X}^4_1-21\bar{X}^2_1\bar{X}_2+6\bar{X}_1\bar{X}_3+3\bar{X}^2_2-\bar{X}_4.
\end{eqnarray*}

Let us now look at the object which motivated much of the present work.

\begin{df}
Let $(\mathcal{A}_k,\phi)$ be a non-commutative $R-k$-probability space. For every $R\in\mathbf{cAlg}_k$, the {\bf $\mathcal{S}_V$-transform} defines a map
$$
\mathcal{S}_V:\mathfrak{G}^1(R)\rightarrow R[[z]]^{\times},
$$
which for $f\in\mathfrak{G}^1(R)$ is given by
\begin{equation}
\label{S-trafo}
\mathcal{S}_V(f):=\frac{1+z}{z}f^{-1}(z),
\end{equation}
and for $a\in\mathcal{A}_k$, such that $\phi(a)\in R^{\times}$, by
\begin{equation}
\label{S-trafo_moment}
\mathcal{S}_V(a):=\frac{1+z}{z}\mathcal{M}^{-1}_a(z),
\end{equation}
where $\mathcal{M}_a^{-1}(z)$ is the compositional inverse.
\end{df}
Note, that in the above definition we should write $(\mathcal{S}_V)_R$ in order to indicate the dependence on the ring $R$, as we are dealing with natural transformations.

The following statements can be deduced from Voiculescu's original theorems, cf.~\cite{V,VDN}. 

\begin{prop}
The $\mathcal{S}_V$-transform defines a natural isomorphism, 
i.e. for every $R\in\mathbf{cAlg}_k$, and the abelian groups $((R[[z]])^{\times},\cdot)$, $(\Lambda(R),\cdot)$, $(\mathfrak{G}^1(R),\boxtimes_V)$  and $(\mathfrak{G}_+^1(R),\boxtimes_V)$, the diagram
\[
\begin{xy}
  \xymatrix{
 (\mathfrak{G}^1_+(R),\boxtimes_V)\ar[r]_{\mathcal{S}_V}^{\sim}  \ar[d]_{\operatorname{incl.} }  & (\Lambda(R),\cdot)\ar[d]_{\operatorname{incl.}}\\
     (\mathfrak{G}^1(R),\boxtimes_V)\ar[r]_{\mathcal{S}_V}^{\sim}       & ((R[[z]])^{\times},\cdot)
               }
\end{xy}
\]
commutes. For  $f,g\in\mathfrak{G}^1(R)$, we have   
$$
\mathcal{S}_V(f\boxtimes_V g)=\mathcal{S}_V(f)\cdot \mathcal{S}_V(g).
$$ 
\end{prop}
\begin{prop} 
Let $(\mathcal{A},\phi)$ be a non-commutative $R-k$-probability space and consider random variables $a,b\in \mathcal{A}$ which are combinatorially free, with $\phi(a),\phi(b)\in R^{\times}$. The $\mathcal{S}_V$-transform of $a\cdot b$ satisfies:
$$
\mathcal{S}_V(a\cdot b)=\mathcal{S}_V(a)\cdot \mathcal{S}_V(b).
$$
\end{prop}

\begin{cor} 
Let $(\mathcal{A},\phi)$ be a non-commutative $R-k$-probability space, and let $\mathcal{A}_{1}:=\phi^{-1}(1)$. For every $a\in\mathcal{A}_{1_R}$, the diagram 
\[
\begin{xy}
  \xymatrix{
   & (\mathcal{A}_{1},\phi)\ar[dl]_{\mathcal{R}}\ar[dr]^{\mathcal{M}}&\\
  (\mathfrak{G}^1_+,\boxtimes)\ar[dr]_{\mathcal{F}}\ar[rr]^{\boxtimes\operatorname{Zeta}_1}&& (\mathfrak{G}^1_+,\boxtimes_V)\ar[dl]^{\mathcal{S}_V}\\
           &  (\Lambda,\cdot)&
               }
\end{xy}
\]
commutes, with the transformations as defined in~(\ref{moment-map}), (\ref{R-trafo}), (\ref{F-trafo}) and (\ref{S-trafo_moment}).
In particular, the relation 
$$
\label{S-trafo_R_series}
\mathcal{R}_{a}^{-1}(z)=(1+z)\cdot\mathcal{M}_{a}^{-1}(z)
$$
holds.
\end{cor}
In the one-dimensional case, as in classical probability theory, the boxed convolution $\boxtimes_V$ can be linearised, as has been proved earlier in~\cite{FMcK2}.

\begin{thm} 
\label{Thm:main_one_dim}For any $\Q$-algebra $R$, the following natural isomorphisms of {\em abelian groups} hold, such that 
$$
\begin{xy}
  \xymatrix{
      (\mathfrak{G}^1_+(R),\boxtimes)\ar[rr]^{\boxtimes\operatorname{Zeta}_1}\ar[dr]_{\mathcal{F}}  &  &   (\mathfrak{G}^1_+(R),\boxtimes_V)    \ar[dl]^{\mathcal{S}_V}\\
                            &(\Lambda(R),\cdot)\ar[d]^{\frac{d}{dz}\log} &  \\
                            &(R[[z]],+)&\\
     (\mathfrak{M}^1(R),+)\ar[rr]^{\boxtimes\operatorname{Zeta}_1}\ar[uuu]_{\operatorname{EXP}}   \ar[ur]^{\id}          &  &    (\mathfrak{M}^1(R),\boxplus_V) \ar[ul]_{\mathcal{R}_V}\ar[uuu]_{\operatorname{EXP}_V}                        }
\end{xy}
$$
commutes, with the maps:
\begin{itemize}
\item $\mathcal{R}_V$: {\bf Voiculescu's $\mathcal{R}_V$-transform} for moments, given by
$$
\mathcal{R}_V(-):=-\boxtimes \operatorname{Moeb}_1,
$$
and which satisfies $\mathcal{R}_V=\mathcal{R}(-)\boxtimes\operatorname{Zeta}_1$.
\item the `Logarithm' for cumulants 
\begin{equation*}
\operatorname{LOG}(-):=\frac{d}{dz}\log(\mathcal{F(-)}),
\end{equation*}
with inverse morphism $\operatorname{EXP}:=\operatorname{LOG}^{-1}$,
\item the `Exponential' morphism for moments
$$
\operatorname{EXP}_V=\operatorname{EXP}(-\boxtimes\operatorname{Zeta}_1)\boxtimes\operatorname{Moeb}_1,
$$
with its inverse morphism $\operatorname{LOG}_V:=\operatorname{EXP}_V^{-1}$, 
\end{itemize}
as shown in the diagram above.
\end{thm}
\begin{cor}
Let  $k$ be a field of characteristic zero, $(\mathcal{A},\phi)$ a non-commutative $R-k$-probability space, and $a,b\in \mathcal{A}$, combinatorially free random variables with distributions $\mu_a$ and $\mu_b$, respectively. The morphism $\operatorname{EXP}_V$ satisfies the ``characteristic properties" of an exponential function:
\begin{itemize}
\item $\operatorname{EXP}_V(\mu_{0_{\mathcal{A}}})=\mu_{1_{\mathcal{A}}}=\operatorname{Zeta}_1$,
\item $\operatorname{EXP}_V(\mu_a\boxplus_V\mu_b)=\operatorname{EXP}_V(\mu_a)\boxtimes_V\operatorname{EXP}_V(\mu_b)$.
\end{itemize}
If additionally we have: $\phi(a)=\phi(b)=1_R$ (and by freeness $\phi(ab)=1_R$), 
then the equality
\begin{equation*}
\label{}
\mu_{ab}=\operatorname{EXP}_V(\operatorname{LOG}_V(\mu_a)\boxplus_V\operatorname{LOG}_V(\mu_b))
\end{equation*}
holds. 
\end{cor}

\begin{cor}
Let  $k$ and $(\mathcal{A},\phi)$ be as above, and $a\in \mathcal{A}$, such that $\phi(a)=1_R$. There exists an isomorphism of additive groups 
\begin{equation*}
\label{}
\mathcal{S}(a)\cong\frac{d}{dz}\log(\mathcal{S}_V(a)).
\end{equation*}
If we assume the natural isomorphism between $R[[z]]$ and $R^{\N}$, and denote by   $\E_{\N}(k)$ the infinite unit matrix with $k$-coefficients, then the matrix $\mathcal{S}(a)$ has the form:
$$
\mathcal{S}(a)=\left(\begin{array}{c|c}1 & d/dz(\log(\mathcal{S}_V(a)))  \\\hline 0&  \E_{\N}(k)\end{array}\right).
$$
\end{cor}
\subsection*{Acknowledgements}
The first author thanks the MPIfM in Bonn for its current hospitality. Both authors thank the MPIfM for a previous opportunity to collaborate on this project.

Authors addresses:\\
Roland Friedrich, 53115 Bonn, Germany\\
rolandf@mathematik.hu-berlin.de\\
John McKay, Dept. Mathematics, Concordia University,\\
Montreal, Canada H3G 1M8\\ mac@mathstat.concordia.ca
\end{document}